\documentclass[11pt,a4paper]{amsart}
\usepackage[utf8]{inputenc}
\usepackage[english]{babel}
\usepackage{amsmath,amsfonts,amssymb,amsthm}
\usepackage{mathtools} 
\usepackage{dsfont} 
\usepackage[dvipsnames]{xcolor}
\usepackage{subcaption} 
\usepackage{pdfsync}
\usepackage{bm,braket,faktor} 
\usepackage{enumitem,booktabs} 
\usepackage{tikz-cd,tikz,pgfplots}
\usetikzlibrary{
	arrows,decorations.pathreplacing,decorations.markings,shapes.geometric,positioning,calc,fadings,angles,patterns,math
}
\tikzfading[name=inner fade, inner color=transparent!0, outer color=transparent!100]
\graphicspath{{/}}
\pgfplotsset{compat=newest}

\usepackage[cal=euler]{mathalfa}
\usepackage{hyperref}
\hypersetup{
	colorlinks = true,
	linkcolor = {Bittersweet},
	citecolor = {ForestGreen},
}

\usepackage[margin=1in]{geometry}
\usepackage{adjustbox}
\makeatletter
\def\l@subsection{\@tocline{2}{0pt}{2.5pc}{5pc}{}}
\makeatother

\usepackage[bottom]{footmisc}

\usepackage[style=alphabetic,maxalphanames=5,
sorting=nyt,%
maxbibnames=99,backend=bibtex]{biblatex}
\renewbibmacro{in:}{}
\usepackage{csquotes}

\usepackage[noabbrev]{cleveref}
\expandafter\def\csname ver@etex.sty\endcsname{3000/12/31}

\crefname{lemma}{lemma}{lemmata}
\Crefname{lemma}{Lemma}{Lemmata}
\crefname{subsection}{subsection}{subsections}
\Crefname{subsection}{Subsection}{Subsections}
\crefname{conjecture}{conjecture}{conjectures}
\Crefname{conjecture}{Conjecture}{Conjectures}

\theoremstyle{plain}
\newtheorem{theorem}{Theorem}[section]
\newtheorem{lemma}[theorem]{Lemma}
\newtheorem{proposition}[theorem]{Proposition}
\newtheorem{corollary}[theorem]{Corollary}

\newtheorem{problem}{Problem}

\theoremstyle{definition}

\newtheorem{remark}[theorem]{Remark}

\theoremstyle{plain}
\newtheorem{introthm}{Theorem}
\newtheorem{introcor}[introthm]{Corollary}

\numberwithin{equation}{section} 

\newcommand{\ri}{\mathrm{i}}
\newcommand{\re}{\mathrm{e}}
\newcommand{\ms}[1]{\mathsf{#1}}

\DeclareMathOperator{\csch}{csch}

\newcommand{\lo}{\mathrm{o}}
\newcommand{\bO}{\mathrm{O}}

\newcommand{\Aut}{\mathrm{Aut}}
\newcommand{\MCG}{\mathrm{MCG}}
\newcommand{\SL}{\mathrm{SL}}
\newcommand{\comb}{\mathrm{comb}}
\newcommand{\sep}{\mathrm{sep}}

\newcommand{\ZZ}{\mathbb{Z}}
\newcommand{\QQ}{\mathbb{Q}}
\newcommand{\RR}{\mathbb{R}}
\newcommand{\CC}{\mathbb{C}}
\newcommand{\EE}{\mathbb{E}}
\newcommand{\PP}{\mathbb{P}}
\newcommand{\GG}{\mathbb{G}}

\newcommand{\cB}{\mathcal{B}}
\newcommand{\cM}{\mathcal{M}}
\newcommand{\cT}{\mathcal{T}}

\begin{document}

\title[Length spectrum of large genus random metric maps]{Length spectrum of large genus random metric maps}
\author[S.~Barazer]{Simon Barazer}
\address[S.~Barazer]{ 
	Université Paris-Saclay, CNRS, IHES, Bures-sur-Yvette, France %
}
\email{simon.barazer@universite-paris-saclay.fr}
\author[A.~Giacchetto]{Alessandro Giacchetto}
\address[A.~Giacchetto]{ 
	Departement Mathematik, ETH Zürich, Rämisstrasse 101, Zürich 8044, Switzerland
}
\email{alessandro.giacchetto@math.ethz.ch
}
\author[M.~Liu]{Mingkun Liu}
\address[M.~Liu]{ 
	DMATH, FSTM, University of Luxembourg, Esch-sur-Alzette, Luxembourg %
}
\email{mingkun.liu@uni.lu}


\begin{abstract}
	We study the length of short cycles on uniformly random metric maps (also known as ribbon graphs) of large genus using a Teichmüller theory approach. We establish that, as the genus tends to infinity, the length spectrum converges to a Poisson point process with an explicit intensity. This result extends the work of Janson and Louf to the multi-faced case.
\end{abstract}

\maketitle

\section{Introduction}
A \emph{map}, or a \emph{ribbon graph}, is a graph with a cyclic ordering of the edges at each vertex. By substituting edges with ribbons and attaching them at each vertex in accordance with the given cyclic order, we create an oriented surface with boundaries on which the graph is drawn (see \Cref{fig:RG}). Since Tutte's pioneering work \cite{Tut63}, ribbon graphs have been extensively studied, partly due to the increased interest following the realisation of their importance in two-dimensional quantum gravity.

\begin{figure}[b]
	\centering
	\begin{tikzpicture}[scale=.6]
		\draw[line width=10pt] (0,-1.7) -- (0,1.7);
		\draw[line width=10pt] (0,0) ellipse (2cm and 1.7cm);
		
		\draw[line width=9pt,white] (0,-1.7) -- (0,1.7);
		\draw[line width=9pt,white] (0,0) ellipse (2cm and 1.7cm);

		\draw[BrickRed,thick] (0,-1.7) -- (0,1.7);
		\draw[BrickRed,thick] (0,0) ellipse (2cm and 1.7cm);

		\node[BrickRed,thick] at (0,1.7) {$\bullet$};
		\node[BrickRed,thick] at (0,-1.7) {$\bullet$};

		\begin{scope}[xshift=8cm]
			\draw[line width=10pt] (0,1.7) to[out=90,in=180] (.5,2.3) to[out=0,in=90] (1,1.7) to[out=-90,in=90] (0,0) -- (0,-.1);
			\draw[line width=9pt,white] (0,1.7) to[out=90,in=180] (.5,2.3) to[out=0,in=90] (1,1.7) to[out=-90,in=90] (0,0) -- (0,-.2);
			\draw[BrickRed,thick] (0,1.7) to[out=90,in=180] (.5,2.3) to[out=0,in=90] (1,1.7) to[out=-90,in=90] (0,0) -- (0,-.2);

			\draw[line width=10pt] (0,-1.7) -- (0,-.1);
			\draw[line width=10pt] (0,0) ellipse (2cm and 1.7cm);
			
			\draw[line width=9pt,white] (0,-1.7) -- (0,0);
			\draw[line width=9pt,white] (0,1.7) to[out=90,in=180] (.5,2.3);
			\draw[line width=9pt,white] (0,0) ellipse (2cm and 1.7cm);
			
			\draw[BrickRed,thick] (0,1.7) to[out=90,in=180] (.5,2.3);
			\draw[BrickRed,thick] (0,0) -- (0,-1.7);

			\draw[BrickRed,thick] (0,0) ellipse (2cm and 1.7cm);
			\node[BrickRed,thick] at (0,1.7) {$\bullet$};
			\node[BrickRed,thick] at (0,-1.7) {$\bullet$};
		\end{scope}
	\end{tikzpicture}
	\caption{A ribbon graph of genus $0$ with $3$ faces (left) and a ribbon graph of genus $1$ with $1$ face (right).}
	\label{fig:RG}
\end{figure}
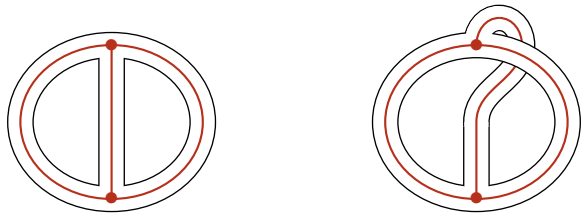

Much attention has been devoted to the study of \emph{metric} maps, i.e.\ ribbon graphs with the assignment of a positive real number to each edge. Remarkably, the moduli space parametrising metric ribbon graphs of a fixed genus $g$ and $n$ faces of fixed lengths is naturally isomorphic to the moduli space of Riemann surfaces of genus $g$ with $n$ punctures \cite{Har86,Pen87,BE88}. This fact was employed by Harer and Zagier to compute the Euler characteristic of the moduli space of Riemann surfaces \cite{HZ86} and by Kontsevich in his proof of Witten's conjecture \cite{Wit91,Kon92}. The latter is a formula that computes the ``number'' of metric ribbon graphs recursively on the Euler characteristic: a topological recursion. The same type of recursion applies to the ``number'' of hyperbolic surfaces, as discovered by Mirzakhani \cite{Mir07a}.

Recently, intensive research efforts have been centred around the \emph{random large genus regime}, both in the combinatorial and in the hyperbolic contexts. The study of large genus asymptotics holds significance for several reasons. Firstly, the intricate nature of several quantities simplifies enormously in the large genus limit, leading to closed-form asymptotic evaluations. Secondly, many interesting quantities associated to several geometric models appear to be exclusively attainable in the asymptotic regime. A (far from exhaustive) list of examples in the combinatorial setting include the connectivity \cite{Ray15,Lou22,BCL}, the local limit \cite{BL21,BL22}, and cycle statistics \cite{JL22,JL23}. Analogously, examples in the hyperbolic setting include the connectivity \cite{Mir13,BCP21,BCP}, the local limit \cite{Mon22}, curve statistics \cite{GPY11,MP19,DGZZ22,DL,WX,HSWX}, and the Laplacian spectrum \cite{WX22,LW23,AM,HM23,LS23,GLST21,Rud23,Nau}.

\subsection{The results}
In the present article, we study short cycles on metric ribbon graphs of large genus. A \emph{cycle} on a metric ribbon graph is sequence of distinct edges that join a sequence of distinct vertices. The length of a cycle is the sum of the lengths of its edges. For a given metric ribbon graph $G$, we define its \emph{length spectrum} $\Lambda(G)$ as the multiset of lengths of all cycles in $G$.

Denote by $\bm{G}_{g,L}$ a uniform random metric ribbon graph of genus $g$ with $n$ marked faces of lengths $L = (L_1, \dots, L_n)$ subject to the scaling condition
\begin{equation} \label{eq:scaling}
	L_1 + \dots + L_n \sim 12g.
\end{equation}
As $\bm{G}_{g,L}$ has almost surely $6g-6+3n$ edges, and the total length of its faces is twice the total length of all edges, the scaling condition implies that, on average, every edge has length $1$. In this sense, the scaling condition is a natural assumption in this context. See \Cref{sec:background} for the precise notion of random metric ribbon graph.

Our main result, proved in \Cref{sec:comb:spec,sec:proof}, is an explicit description of $\Lambda(\bm{G}_{g,L})$ in the large genus limit as a \emph{Poisson point process}.

\begin{introthm} \label{thm:PPP}
	For any fixed $n$, the random multiset $\Lambda(\bm{G}_{g,L})$, viewed as a point process on $\RR_{\geq 0}$, converges in distribution as $g \to \infty$ to a Poisson point process of intensity $\lambda$ defined by
	\begin{equation} \label{eq:intensity}
		\lambda(\ell)
		\coloneqq
		\frac{\cosh(\ell) - 1}{\ell}.
	\end{equation}
\end{introthm}

Actually, we shall prove the following result which implies \Cref{thm:PPP} through the method of moments. For any non-empty interval $[a,b) \subset \RR_{\ge 0}$, denote by $N_{[a,b)}(\bm{G}_{g,L})$ the number of cycles in $\bm{G}_{g,L}$ of length falling within the interval $[a,b)$.

\begin{introthm} \label{thm:main}
	For any fixed $n$ and disjoint intervals $[a_1,b_1), \dots, [a_p,b_p) \subset \RR_{\ge 0}$, the random vector
	\begin{equation}
		\Bigl(
			N_{[a_1, b_1)}(\bm{G}_{g,L}),
			\dots,
			N_{[a_p, b_p)}(\bm{G}_{g,L})
		\Bigr)
	\end{equation}
	converges in distribution as $g \to \infty$ to a vector of independent Poisson variables of means
	\begin{equation}
		\left(
			\int_{a_1}^{b_1} \lambda(\ell) \, d\ell,
			\dots,
			\int_{a_p}^{b_p} \lambda(\ell) \, d\ell
		\right) .
	\end{equation}
\end{introthm}

As an application, we obtain the law of the length of the shortest cycle, known as \emph{girth} or \emph{systole}, on a random metric ribbon graph of large genus.

\begin{introcor} \label{cor:girth}
	For any fixed $n$, the girth converges in distribution to a non-homogeneous exponential distribution with rate function $\lambda$. In other words, we have
	\begin{equation}
		\lim_{g \to \infty} \PP\bigl[ \mathrm{girth}(\bm{G}_{g,L}) \leq t \bigr]
		=
		1 - \exp \mathopen{}\left( - \int_0^t \lambda(\ell) \, d\ell \right)\mathclose{}.
	\end{equation}
\end{introcor}

We remark that all results presented here hold in the more general setting where the boundary lengths are subjected to the scaling condition $L_1 + \dots + L_n \sim \mu 12g$ for some $\mu > 0$. In this case, the intensity is given by by $\lambda_{\mu}(\ell) \coloneqq (\cosh(\ell/\mu) - 1)/\ell$. Besides, all results are still valid when replacing ``cycles'' by ``closed walks that do not traverse the same edge more than $D$ times'', with $D$ a fixed positive integer.

In comparison to the analogous results for hyperbolic surfaces due to Mirzakhani and Petri \cite{MP19}, a natural comment is due. The length spectrum of a hyperbolic surface (or more generally, any Riemannian manifold) is commonly defined as the multiset of lengths of the shortest primitive closed curve in each free homotopy class. In this combinatorial setting, it would make sense to consider the length spectrum defined analogously, instead of restricting it to only cycles. However, \Cref{thm:main} fails to hold true when all closed curves are considered, and it fails even when restricted to all simple closed curves.
More precisely, let $\bar{N}_{[a,b)}(G)$ denote the number of (free homotopy classes of primitive) closed curves in $G$ with length falling within the interval $[a,b)$, and let $\bar{N}_{[a,b)}^\circ(G)$ denote the number of simple closed curves in $G$.

\begin{introthm} \label{thm:turtle:neck}
	For any fixed $(g,n)$, boundary lengths $L \in \RR_{>0}^n$, and $[a,b) \subset \RR_{\ge 0}$, we have:
	\begin{itemize} \setlength\itemsep{.5em}
		\item $\EE\bigl[ \bar{N}_{[a, b)}(\bm{G}_{g,L}) \bigr] = \infty$,

		\item $\EE\bigl[ \bar{N}^{\circ}_{[a,b)}(\bm{G}_{g,L}) \bigr] < \infty$ and  $\EE\bigl[ \bar{N}^{\circ}_{[a,b)}(\bm{G}_{g,L})^k \bigr] = \infty$ for any $k > 3/2$.
	\end{itemize}
\end{introthm}

The intuitive reason behind the above failing is the fact that metric ribbon graphs have zero curvature, and as such they can be scaled. Hence, the number of short closed curves on a metric ribbon graph $G$ can blow up when $G$ approaches the boundary of the moduli space. See \Cref{sec:wrong} for a detailed discussion.

\Cref{thm:PPP} is supported by evidence from numerical simulations, as discussed in \Cref{sec:numerics}. \Cref{fig:simulation} illustrates cycle length statistics derived from three samples of $10^3$ uniform random one-faced metric ribbon graphs of genera $2$, $8$, and $64$ respectively. The theoretical prediction is depicted in lime.

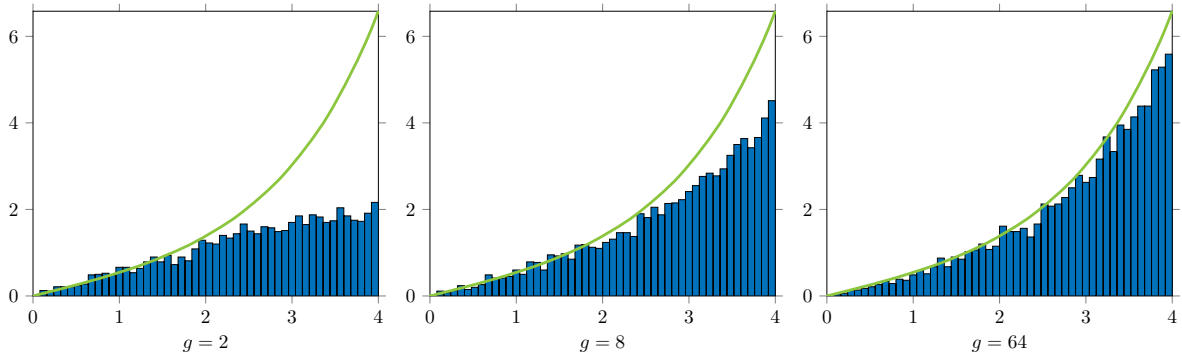
\begin{figure}
	\centering
	\begin{subfigure}[b]{.32\textwidth}
		\centering
        \begin{adjustbox}{width=\textwidth}
        \begin{tikzpicture}
		\begin{axis}[
			xmin=0, xmax=4, ymin=0, ymax=6.58,
			tick align=outside,
			xlabel = {\(g = 2\)}
			]
			\addplot[ybar interval,mark=no,fill=RoyalBlue] plot coordinates {
				( 0.08 , 0.1250 ) ( 0.16 , 0.1125 ) ( 0.24 , 0.2125 ) ( 0.32 , 0.2125 ) ( 0.40 , 0.2125 ) ( 0.48 , 0.2375 ) ( 0.56 , 0.2750 ) ( 0.64 , 0.4875 ) ( 0.72 , 0.5000 ) ( 0.80 , 0.5250 ) ( 0.88 , 0.4875 ) ( 0.96 , 0.6625 ) ( 1.04 , 0.6625 ) ( 1.12 , 0.5375 ) ( 1.20 , 0.6375 ) ( 1.28 , 0.7875 ) ( 1.36 , 0.9000 ) ( 1.44 , 0.7875 ) ( 1.52 , 0.9375 ) ( 1.60 , 0.7250 ) ( 1.68 , 0.9000 ) ( 1.76 , 0.8125 ) ( 1.84 , 1.0875 ) ( 1.92 , 1.2875 ) ( 2.00 , 1.2250 ) ( 2.08 , 1.2000 ) ( 2.16 , 1.4000 ) ( 2.24 , 1.3375 ) ( 2.32 , 1.4375 ) ( 2.40 , 1.6625 ) ( 2.48 , 1.5000 ) ( 2.56 , 1.4375 ) ( 2.64 , 1.6000 ) ( 2.72 , 1.5750 ) ( 2.80 , 1.4875 ) ( 2.88 , 1.5125 ) ( 2.96 , 1.7000 ) ( 3.04 , 1.8500 ) ( 3.12 , 1.6500 ) ( 3.20 , 1.8750 ) ( 3.28 , 1.8250 ) ( 3.36 , 1.7000 ) ( 3.44 , 1.7375 ) ( 3.52 , 2.0375 ) ( 3.60 , 1.8500 ) ( 3.68 , 1.7500 ) ( 3.76 , 1.7250 ) ( 3.84 , 1.9125 ) ( 3.92 , 2.1625 ) ( 4.00 , 2.2500 )
			};
			\addplot[LimeGreen, ultra thick, smooth, samples=20] {(cosh(x) - 1)/x};
		\end{axis}
		\end{tikzpicture}
        \end{adjustbox}
	\end{subfigure}
	\begin{subfigure}[b]{.32\textwidth}
		\centering
        \begin{adjustbox}{width=\textwidth}
		\begin{tikzpicture}
		\begin{axis}[
			xmin=0, xmax=4, ymin=0, ymax=6.58,
			tick align=outside,
			xlabel = {\(g = 8\)}
			]
			\addplot[ybar interval,mark=no,fill=RoyalBlue] plot coordinates {
				( 0.08 , 0.1125 ) ( 0.16 , 0.0750 ) ( 0.24 , 0.1250 ) ( 0.32 , 0.2375 ) ( 0.40 , 0.1500 ) ( 0.48 , 0.2000 ) ( 0.56 , 0.2625 ) ( 0.64 , 0.4875 ) ( 0.72 , 0.4125 ) ( 0.80 , 0.4125 ) ( 0.88 , 0.4500 ) ( 0.96 , 0.6000 ) ( 1.04 , 0.5000 ) ( 1.12 , 0.7875 ) ( 1.20 , 0.7750 ) ( 1.28 , 0.6000 ) ( 1.36 , 0.9500 ) ( 1.44 , 0.9250 ) ( 1.52 , 0.9750 ) ( 1.60 , 0.8500 ) ( 1.68 , 1.1750 ) ( 1.76 , 1.1875 ) ( 1.84 , 1.1250 ) ( 1.92 , 1.1000 ) ( 2.00 , 1.2375 ) ( 2.08 , 1.3125 ) ( 2.16 , 1.4625 ) ( 2.24 , 1.4625 ) ( 2.32 , 1.3750 ) ( 2.40 , 1.9000 ) ( 2.48 , 1.8125 ) ( 2.56 , 2.0500 ) ( 2.64 , 1.8750 ) ( 2.72 , 2.1375 ) ( 2.80 , 2.1500 ) ( 2.88 , 2.2250 ) ( 2.96 , 2.4125 ) ( 3.04 , 2.5500 ) ( 3.12 , 2.7625 ) ( 3.20 , 2.8375 ) ( 3.28 , 2.7750 ) ( 3.36 , 2.9375 ) ( 3.44 , 3.2500 ) ( 3.52 , 3.5000 ) ( 3.60 , 3.6375 ) ( 3.68 , 3.4250 ) ( 3.76 , 3.6625 ) ( 3.84 , 4.1125 ) ( 3.92 , 4.5125 ) ( 4.00 , 4.3625 )
			};
			\addplot[LimeGreen, ultra thick, smooth, samples=20] {(cosh(x) - 1)/x};
		\end{axis}
		\end{tikzpicture}
        \end{adjustbox}
	\end{subfigure}
	\begin{subfigure}[b]{.32\textwidth}
		\centering
        \begin{adjustbox}{width=\textwidth}
		\begin{tikzpicture}
		\begin{axis}[
			xmin=0, xmax=4, ymin=0, ymax=6.58,
			tick align=outside,
			xlabel = {\(g = 64\)}
			]
			\addplot[ybar interval,mark=no,fill=RoyalBlue] plot coordinates {
				( 0.08 , 0.0500 ) ( 0.16 , 0.0625 ) ( 0.24 , 0.1250 ) ( 0.32 , 0.1375 ) ( 0.40 , 0.1750 ) ( 0.48 , 0.2125 ) ( 0.56 , 0.2625 ) ( 0.64 , 0.3375 ) ( 0.72 , 0.2875 ) ( 0.80 , 0.3875 ) ( 0.88 , 0.3625 ) ( 0.96 , 0.4875 ) ( 1.04 , 0.5875 ) ( 1.12 , 0.5125 ) ( 1.20 , 0.6750 ) ( 1.28 , 0.8750 ) ( 1.36 , 0.6750 ) ( 1.44 , 0.9000 ) ( 1.52 , 0.8500 ) ( 1.60 , 1.0250 ) ( 1.68 , 1.0625 ) ( 1.76 , 1.2000 ) ( 1.84 , 1.0750 ) ( 1.92 , 1.1500 ) ( 2.00 , 1.6125 ) ( 2.08 , 1.4875 ) ( 2.16 , 1.4875 ) ( 2.24 , 1.5625 ) ( 2.32 , 1.3625 ) ( 2.40 , 1.6625 ) ( 2.48 , 2.1250 ) ( 2.56 , 2.0750 ) ( 2.64 , 2.1250 ) ( 2.72 , 2.2750 ) ( 2.80 , 2.5000 ) ( 2.88 , 2.7875 ) ( 2.96 , 2.6250 ) ( 3.04 , 2.7375 ) ( 3.12 , 3.1625 ) ( 3.20 , 3.6750 ) ( 3.28 , 3.3375 ) ( 3.36 , 3.9500 ) ( 3.44 , 3.8500 ) ( 3.52 , 4.1375 ) ( 3.60 , 4.3875 ) ( 3.68 , 4.3875 ) ( 3.76 , 5.2250 ) ( 3.84 , 5.2875 ) ( 3.92 , 5.5875 ) ( 4.00 , 5.7875 )
			};
			\addplot[LimeGreen, ultra thick, smooth, samples=20] {(cosh(x) - 1)/x};
		\end{axis}
		\end{tikzpicture}
        \end{adjustbox}
	\end{subfigure}
	\caption{
		In blue, the cycle length statistics of random  unicellular metric maps of genus $g = 2$, $8$, and $64$, sampled over $10^3$ units and properly rescaled. The predicted intensity $\lambda$ is depicted in lime.
	}
	\label{fig:simulation}
\end{figure}

\subsection{Related works and proof strategy}
In \cite{MP19}, Mirzakhani and Petri study the length spectrum of a random closed hyperbolic surface $\bm{S}_g$ of genus $g$, sampled according to the Weil--Petersson measure. They prove that the length spectrum $\Lambda(\bm{S}_g)$ converges in distribution as $g \to \infty$ to a Poisson point process with intensity $\lambda$ defined as in \Cref{eq:intensity}. In \cite{JL23}, Janson and Louf consider uniform random unicellular (i.e.\ one-faced) maps $\bm{U}_{v,g}$ of genus $g$ with $v$ vertices, metrised by assigning length $1$ to all edges. They prove that, as $g$, $v \to \infty$ with $g = \lo(v)$, the normalised length spectrum $\sqrt{12g/v} \cdot \Lambda(\bm{U}_{v,g})$ converges in distribution to a Poisson point process with exactly the same intensity $\lambda$. The convergence of the length spectrum of large random hyperbolic surfaces or random maps to a Poisson point process is not entirely unexpected (such events follow the \textit{Poisson paradigm}, see \cite{Wor99,MWW04,Pet17,Roi} for results along these lines). However, the precise matching of intensity functions is somehow miraculous.

While both \cite{MP19} and \cite{JL23} employ the moment method in their proofs, their approaches are of completely different natures. On the one hand, Mirzakhani and Petri compute moments by means of an integration formula developed by Mirzakhani in her thesis \cite{Mir07a}, and the large genus asymptotic analysis relies on the work of Mirzakhani and Zograf on Weil--Petersson volumes \cite{Mir13,MZ15}. On the other hand, the approach adopted by Janson and Louf is entirely combinatorial. In their proof, a bijection due to Chapuy, Féray, and Fusy \cite{CFF13} between unicellular maps and trees decorated by permutations plays a crucial role, enabling them to proceed using results on random trees and random permutations, both extensively explored subjects. We emphasise that, as the Chapuy--Féray--Fusy bijection is tied to the unicellular case, the method of Janson and Louf does not extend to the multi-faced case.

The current paper follows a Teichmüller theory approach, and the proof strategy is similar to that of \cite{MP19}. More precisely, our proof makes use of the framework established in \cite{ABCGLW}, which brings several tools from hyperbolic geometry into combinatorics, as well as the recent result by Aggarwal \cite{Agg21} on the large genus asymptotics of $\psi$-class intersection numbers. This combination of techniques allows us to extend the results of \cite{JL23} to the multi-faced cases.

Let us briefly mention why the model considered in \cite{JL23} coincides with the one discussed in the present work for $n = 1$. Curien, Janson, Kortchemski, Louf, and Marzouk proposed an alternative model to random metric unicellular maps which behave as $\bm{U}_{v,g}$ (see \cite{Lou23}).
Let $\bm{V}_g$ be a uniform random unicellular trivalent map of genus $g$, metrised by assigning to all $6g-3$ edges i.i.d.\ $\mathrm{Exp}(1)$ random lengths $\ell_i$ (exponential distribution of parameter $1$). Denote by $L \coloneqq 2(\ell_1 + \cdots + \ell_{6g-3})$ the length of the unique face. It is a standard fact that if $\ell_1, \dots, \ell_{6g-3}$ are i.i.d.\ $\mathrm{Exp}(1)$ variables, then the random vector $X \coloneqq(2\ell_1/L, \dots, 2\ell_{6g-3}/L)$ is $\mathrm{Dir}(1^{6g-3})$ distributed (Dirichlet distribution of order $6g-3$ of parameters $(1, 1, \dots, 1)$), and $L$ and $X$ are independent thanks to Lukács's proportion-sum independence theorem. Such a model, conditioned on $L$ being fixed, is nothing but $\bm{G}_{g,L}$.

\subsection{Outlook}
As highlighted in \cite[Section~1.4]{JL23}, random metric ribbon graphs and random hyperbolic surfaces exhibit strikingly similar behaviours. The former has the advantage of being both combinatorial and topological in nature, bringing many tools and insights from combinatorics and graph theory into the realm of 2D geometry. Besides, ribbon graphs are prone to numerical tests, as demonstrated in \Cref{fig:simulation}.

These advantages were previously emphasised in \cite{ABCGLW}, where, for instance, the computation of a combinatorial versions of Mirzakhani's kernels $\mathcal{R}$ and $\mathcal{D}$ reduced to a straightforward combinatorial check rather than intricate hyperbolic geometry computations. Another illustration is the computation of the expectation value of cycles with length in $[a,b)$ and self-intersection one, where a simple combinatorial consideration implies that in the $g \to \infty$ limit, the expectation value reads
\begin{equation}
	\frac{1}{g}
	\int_a^b \Bigl(
		\tfrac{1}{3} \, \ell \cosh(\ell)
		-
		\tfrac{2}{3} \, \sinh(\ell)
		-
		\tfrac{11}{12} \, \ell \cosh \mathopen{}\left( \tfrac{\ell}{2} \right)\mathclose{}
		+
		\left(
			\tfrac{\ell^2}{96} + 4
		\right)
		\sinh \mathopen{}\left( \tfrac{\ell}{2} \right)\mathclose{}
		-
		\tfrac{4}{3} \, \ell
	\Bigr) d\ell
	+
	\bO \mathopen{}\left( \frac{1}{g^2} \right)\mathclose{}.
\end{equation}
This is an example of a Friedman--Ramanujan function, a class of functions that play a central role in the spectral gap problem in the hyperbolic setting \cite{AM}.

A plausible explanation for the similarity of the two models could be derived from the spine construction of \cite{BE88}. This direction has already been investigated for various quantities associated with both the hyperbolic and combinatorial moduli spaces, such as the symplectic structures in \cite{Do}, naturally defined functions in \cite{ABCGLW}, and shapes of complementary subsurfaces in \cite{AC}. We intend to revisit this direction in future works.

With this perspective in mind, it is then natural to consider the following question.

\begin{problem}
	What does the Laplacian spectrum of $\bm{G}_{g,L}$ look like as $g \to \infty$?
	And how does it relate to the Laplacian spectrum of random hyperbolic surfaces?
\end{problem}

To tackle this problem, we may start by identifying the \emph{local limit}, also known as the \emph{Benjamini--Schramm limit} \cite{BS01}, of $\bm{G}_{g,L}$ when $g \to \infty$.
It is well-known that the spectra of a sequence of graphs are closely related to the Benjamini--Schramm limit of the sequence.
More specifically, the Benjamini--Schramm convergence implies the convergence of spectral measures, see e.g. \cite{ATV}.
This principle has recently been extended to metric graphs (or quantum graphs) in \cite{AS19, AISW}. Hence, we consider the following question of particular interest.

\begin{problem}
	Understand the local limit of $\bm{G}_{g,L}$ as $g \to \infty$.
\end{problem}

It is reasonable to expect that the local limit of a random trivalent unicellular map is the random infinite trivalent metric tree $\bm{T}$ with i.i.d.\ $\mathrm{Exp}(1)$ distributed edge lengths. Thus, we naturally conjecture that $\bm{G}_{g,L}$ converges in the Benjamini--Schramm sense to $\bm{T}$. With this in mind, it would be interesting to explore in parallel the following question.

\begin{problem}
	Understand the Laplacian spectrum of $\bm{T}$. How does it compare to the Laplacian spectrum of the hyperbolic plane?
\end{problem}

The local limit of unicellular maps (not necessarily trivalent) when the genus grows in proportion to the number of edges has been identified by Angel, Chapuy, Curien and Ray in \cite{ACCR13}, which is a supercritial Galton--Watson tree conditional to be infinite.
More recently, the case of moderate genus growth has been studied by Curien, Kortchemski, and Marzouk \cite{CKM22}.
In particular, they prove that the mesoscopic scaling limit of the core of such maps is the infinite trivalent tree whose edge lengths are i.i.d.\ exponential variables.

\subsection*{Acknowledgement}
The authors would like to thank Nalini~Anantharaman, Jérémie~Bouttier, Nicolas~Curien, Vincent~Delecroix, David~Fisac, Sébastien~Labbé, Baptiste~Louf, Nina~Morishige, Hugo~Parlier, Bram~Petri, Yunhui~Wu, and Anton~Zorich for helpful discussions.
The last author would like to thank Camille~Deperraz, Louis~Deperraz, Thierry~Deperraz, and Coralie~Oudot for their warm hospitality, making it possible to visit the other two authors multiple times in Paris.
This project started during a ``mini-rencontre ANR MoDiff''; it is a pleasure to thank its organisers.

S.B.\ and A.G.\ are supported by the ERC-SyG project ``Recursive and Exact New Quantum Theory'' (ReNewQuantum), which received funding from the European Research Council under the European Union's Horizon 2020 research and innovation programme under grant agreement N\textsuperscript{\underline{\scriptsize o}} 810573.
M.L.\ is supported by the Luxembourg National Research Fund OPEN grant O19/13865598.

\section{Background}
\label{sec:background}
In this section, we recall some background material about the geometry of the combinatorial Teichmüller and moduli spaces (see \cite{ABCGLW} for more details), as well as some probabilistic tools that will be used in order to prove the main result of the paper.

\subsection{Combinatorial Teichmüller and moduli spaces}
A ribbon graph is a finite graph $G$ together with a cyclic order of the edges at each vertex. By replacing each edge by a closed ribbon and glueing them at each vertex according to the cyclic order, we obtain a topological, oriented, compact surface called the geometric realisation of $G$. Notice that the graph is a deformation retract of its geometric realisation. We will assume that $G$ is connected and all vertices have valency $\ge 3$.

The geometric realisation of a ribbon graph $G$ will have $n \ge 1$ boundary components, also called faces, and we always assume they are labelled as $\partial_1 G, \dots , \partial_n G$. Denote by $V(G)$, $E(G)$, $F(G)$ the set of vertices, edges, and faces, respectively. We define the genus $g \ge 0$ to be the genus of the geometric realisation. Thus, $|V(G)| - |E(G)| + |F(G)| = 2 - 2g$. The datum $(g, n)$ is called the type of $G$.

A metric ribbon graph is the data $(G,\ell)$ of a ribbon graph $G$ together with the assignment of a positive real number for each edge, that is, $\ell \colon E(G) \to \RR_{>0}$. In the following, we will omit the map $\ell$ from the notation, and simply denote it by $\ell_G$ when needed. For a given metric ribbon graph $G$ and a non-trivial edge-path $\gamma$, we can define the length $\ell_G(\gamma)$ as the sum of the length of edges (with multiplicity) visited by $\gamma$. In particular, we can talk about length of the boundary components $\ell_G(\partial_i G)$.

Fix now a connected, compact, oriented surface $\Sigma$ of genus $g \ge 0$ with $n \ge 1$ labelled boundaries, denoted $\partial_1\Sigma, \dots , \partial_n\Sigma$. Fix $L \in \RR_{>0}^n$. Define the combinatorial Teichmüller space as the space parametrising metric ribbon graphs of type $(g,n)$ with fixed boundary lengths embedded into $\Sigma$, up to isotopy:
\begin{equation}
	\cT^{\textup{comb}}_{\Sigma}(L)
	\coloneqq
	\Set{
		(G,f) | \substack{
			\displaystyle G \text{ is a metric ribbon graph} \\[.25em]
			\displaystyle \text{with boundary lengths }\ell_G(\partial_i G) = L_i \\[.25em]
			\displaystyle f \colon G \hookrightarrow \Sigma \text{ is a retract}
		}
	} \Bigg/ \sim
\end{equation}
where the equivalence relation is given by
\begin{equation}
	(G,f) \sim (G',f')
	\qquad\text{if and only if}\qquad
	\substack{
		\displaystyle \exists \, \varphi \colon  G \to G' \text{ an isometry such that } f' \circ \varphi = f \\
		\displaystyle \text{and $f$ and $f'$ are isotopic.}
	}
\end{equation}
Notice that, as $G$ is a retract of $\Sigma$, it has the same genus and number of boundary components as $\Sigma$. Often, we will denote elements of $\cT^{\textup{comb}}_{\Sigma}(L)$ by $\GG$.

It can be shown that $\cT^{\textup{comb}}_{\Sigma}(L)$ is a real polytopal complex of dimension $6g - 6 + 2n$. The cells are labelled by embedded ribbon graphs, and they parametrise all possible metrics with fixed boundary lengths on the corresponding embedded graph; the cells are glued together via edge degeneration (see \Cref{fig:comb:Teich:torus} for an example). The pure mapping class group $\MCG_{\Sigma}$ of isotopy classes of orientation preserving homeomorphisms of $\Sigma$ preserving the boundary components naturally acts on $\cT^{\textup{comb}}_{\Sigma}(L)$. The quotient space
\begin{equation}
	\cM^{\textup{comb}}_{g,n}(L)
	\coloneqq
	\cT^{\textup{comb}}_{\Sigma}(L) \big/ \MCG_{\Sigma}
\end{equation}
is called the combinatorial moduli space. It parametrises metric ribbon graphs of type $(g,n)$ with fixed boundary lengths. It is a real polytopal orbicomplex of dimension $6g - 6 + 2n$. The orbicells are labelled by ribbon graphs, and they parametrise all possible metrics with fixed boundary lengths on the corresponding ribbon graph, up to automorphism (see \Cref{fig:comb:mod:torus} for an example).

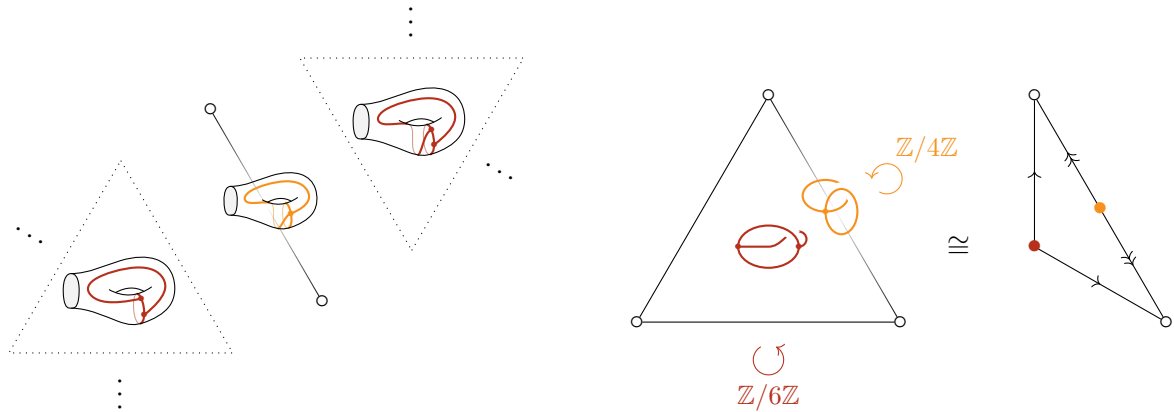
\begin{figure}
	\centering
	\begin{subfigure}[t]{.45\textwidth}
	\centering
	\begin{tikzpicture}[scale=1.7]
		
		\draw [dotted] (90:1) -- (-30:1) -- (-150:1) -- cycle;

		\node [rotate=90] at (-90:.8) {$\cdots$};
		\node [rotate=-30] at (150:.8) {$\cdots$};

		\begin{scope}[x=1pt,y=1pt,scale=.12,xshift=-7.7cm,yshift=-26cm]
			\draw [fill,fill opacity=.05] (128, 768) .. controls (112, 768) and (112, 736) .. (114.6667, 720) .. controls (117.3333, 704) and (122.6667, 704) .. (128, 704) .. controls (133.3333, 704) and (138.6667, 704) .. (141.3333, 720) .. controls (144, 736) and (144, 768) .. (128, 768);
			\draw (209.404, 732.847) .. controls (228.388, 743.2543) and (245.7957, 744.001) .. (261.627, 735.087);
			\draw [BrickRed,thick] (248.106, 730.251) .. controls (258.861, 733.847) and (269.878, 680.516) .. (254.529, 675.325);
			\draw [BrickRed,opacity=.5] (254.529, 675.325) .. controls (241.304, 673.617) and (238.33, 728.644) .. (248.106, 730.251);
			\draw [BrickRed,thick] (262.033, 691.578) .. controls (304.739, 725.931) and (304.677, 745.5385) .. (296.6465, 759.1363) .. controls (288.616, 772.734) and (272.617, 780.322) .. (247.5522, 779.2083) .. controls (222.4873, 778.0947) and (188.3567, 768.2793) .. (171.6688, 757.4658) .. controls (154.981, 746.6523) and (155.736, 734.8407) .. (166.9005, 725.4618) .. controls (178.065, 716.083) and (199.639, 709.137) .. (256.785, 721.802);
			\draw (197.435, 738.371) .. controls (224.2723, 723.5677) and (250.4063, 725.2853) .. (275.837, 743.524);
			\draw (128, 768) .. controls (160, 768) and (192, 784) .. (218.6667, 792) .. controls (245.3333, 800) and (266.6667, 800) .. (282.6667, 794.6667) .. controls (298.6667, 789.3333) and (309.3333, 778.6667) .. (314.6667, 762.6667) .. controls (320, 746.6667) and (320, 725.3333) .. (309.3333, 709.3333) .. controls (298.6667, 693.3333) and (277.3333, 682.6667) .. (261.3333, 677.3333) .. controls (245.3333, 672) and (234.6667, 672) .. (213.3333, 680) .. controls (192, 688) and (160, 704) .. (128, 704);
			\node [BrickRed] at (256.45, 721.746) {\scalebox{.5}{$\bullet$}};
			\node [BrickRed] at (262.033, 691.578) {\scalebox{.5}{$\bullet$}};
		\end{scope}
		
		
		\draw [black,path fading=east] (30:1.3) -- ($(30:1.3) + (120:0.866)$);
		\draw [black,path fading=west] (30:1.3) -- ($(30:1.3) + (-60:0.866)$);

		\node [white] at ($(30:.8) + (90:1)$) {$\bullet$};
		\node [white] at ($(30:.8) + (-30:1)$) {$\bullet$};
		\node at ($(30:.8) + (90:1)$) {$\circ$};
		\node at ($(30:.8) + (-30:1)$) {$\circ$};

		\begin{scope}[x=1pt,y=1pt,scale=.1,xshift=4cm,yshift=-19cm]
			\draw [BurntOrange,thick] (248.106, 730.251) .. controls (258.861, 733.847) and (269.878, 680.516) .. (254.529, 675.325);
			\draw [BurntOrange,opacity=.5] (254.529, 675.325) .. controls (241.304, 673.617) and (238.33, 728.644) .. (248.106, 730.251);
			\draw [BurntOrange,thick] (260.839, 707.65) .. controls (304.739, 725.931) and (304.677, 745.5385) .. (296.6465, 759.1363) .. controls (288.616, 772.734) and (272.617, 780.322) .. (247.5522, 779.2083) .. controls (222.4873, 778.0947) and (188.3567, 768.2793) .. (171.6688, 757.4658) .. controls (154.981, 746.6523) and (155.736, 734.8407) .. (168.1023, 731.8956) .. controls (180.4685, 728.9505) and (204.446, 734.872) .. (213.157, 731.563);
			\draw [BurntOrange,opacity=.5] (213.157, 731.563) .. controls (223.556, 729.146) and (222.7375, 716.9415) .. (222.474, 703.8698) .. controls (222.2105, 690.798) and (222.502, 676.859) .. (231.579, 674.591);
			\draw [BurntOrange,thick] (231.579, 674.591) .. controls (239.505, 673.659) and (247.229, 704.233) .. (260.839, 707.65);

			\draw [fill,fill opacity=.05] (128, 768) .. controls (112, 768) and (112, 736) .. (114.6667, 720) .. controls (117.3333, 704) and (122.6667, 704) .. (128, 704) .. controls (133.3333, 704) and (138.6667, 704) .. (141.3333, 720) .. controls (144, 736) and (144, 768) .. (128, 768);
			\draw (209.404, 732.847) .. controls (228.388, 743.2543) and (245.7957, 744.001) .. (261.627, 735.087);
			\draw (197.435, 738.371) .. controls (224.2723, 723.5677) and (250.4063, 725.2853) .. (275.837, 743.524);
			\draw (128, 768) .. controls (160, 768) and (192, 784) .. (218.6667, 792) .. controls (245.3333, 800) and (266.6667, 800) .. (282.6667, 794.6667) .. controls (298.6667, 789.3333) and (309.3333, 778.6667) .. (314.6667, 762.6667) .. controls (320, 746.6667) and (320, 725.3333) .. (309.3333, 709.3333) .. controls (298.6667, 693.3333) and (277.3333, 682.6667) .. (261.3333, 677.3333) .. controls (245.3333, 672) and (234.6667, 672) .. (213.3333, 680) .. controls (192, 688) and (160, 704) .. (128, 704);
			\node [BurntOrange] at (260.8388, 707.6496) {\scalebox{.5}{$\bullet$}};
		\end{scope}

		
		\draw [dotted] ($(30:2.6) + (-90:1)$) -- ($(30:2.6) + (30:1)$) -- ($(30:2.6) + (150:1)$) -- cycle;

		\node [rotate=90] at ($(30:2.6) + (90:.8)$) {$\cdots$};
		\node [rotate=-30] at ($(30:2.6) + (-30:.8)$) {$\cdots$};

		\begin{scope}[x=1pt,y=1pt,scale=.12,xshift=11cm,yshift=-15cm]
			\draw [BrickRed,thick] (248.106, 730.251) .. controls (258.861, 733.847) and (269.878, 680.516) .. (254.529, 675.325);
			\draw [BrickRed,opacity=.5] (254.529, 675.325) .. controls (241.304, 673.617) and (238.33, 728.644) .. (248.106, 730.251);
			\draw [BrickRed,thick] (262.091, 693.38) .. controls (304.739, 725.931) and (304.677, 745.5385) .. (296.6465, 759.1363) .. controls (288.616, 772.734) and (272.617, 780.322) .. (247.5522, 779.2083) .. controls (222.4873, 778.0947) and (188.3567, 768.2793) .. (171.6688, 757.4658) .. controls (154.981, 746.6523) and (155.736, 734.8407) .. (168.1023, 731.8956) .. controls (180.4685, 728.9505) and (204.446, 734.872) .. (213.157, 731.563);
			\draw [BrickRed,opacity=.5] (213.157, 731.563) .. controls (223.556, 729.146) and (222.7375, 716.9415) .. (222.474, 703.8698) .. controls (222.2105, 690.798) and (222.502, 676.859) .. (231.579, 674.591);
			\draw [BrickRed,thick] (231.579, 674.591) .. controls (239.505, 673.659) and (242.7, 717.524) .. (257.365, 720.38);

			\draw [fill,fill opacity=.05] (128, 768) .. controls (112, 768) and (112, 736) .. (114.6667, 720) .. controls (117.3333, 704) and (122.6667, 704) .. (128, 704) .. controls (133.3333, 704) and (138.6667, 704) .. (141.3333, 720) .. controls (144, 736) and (144, 768) .. (128, 768);
			\draw (209.404, 732.847) .. controls (228.388, 743.2543) and (245.7957, 744.001) .. (261.627, 735.087);
			\draw (197.435, 738.371) .. controls (224.2723, 723.5677) and (250.4063, 725.2853) .. (275.837, 743.524);
			\draw (128, 768) .. controls (160, 768) and (192, 784) .. (218.6667, 792) .. controls (245.3333, 800) and (266.6667, 800) .. (282.6667, 794.6667) .. controls (298.6667, 789.3333) and (309.3333, 778.6667) .. (314.6667, 762.6667) .. controls (320, 746.6667) and (320, 725.3333) .. (309.3333, 709.3333) .. controls (298.6667, 693.3333) and (277.3333, 682.6667) .. (261.3333, 677.3333) .. controls (245.3333, 672) and (234.6667, 672) .. (213.3333, 680) .. controls (192, 688) and (160, 704) .. (128, 704);
			\node [BrickRed] at (262.0913, 693.3805) {\scalebox{.5}{$\bullet$}};
			\node [BrickRed] at (257.3645, 720.38) {\scalebox{.5}{$\bullet$}};
		\end{scope}	
	\end{tikzpicture}
	\subcaption{The combinatorial Teichmüller space of a one-holed torus. It has infinitely many cells 2D cells and 1D cells, corresponding to the different embeddings of the two ribbon graphs of type $(1,1)$.}
	\label{fig:comb:Teich:torus}
	\end{subfigure}
	\hfill%
	\begin{subfigure}[t]{.5\textwidth}
	\centering
	\begin{tikzpicture}[scale=1]
		\draw [black,path fading=west] (-30:2) -- (30:1);
		\draw [black,path fading=east] (90:2) -- (30:1);
		\draw (90:2) -- (-150:2) -- (-30:2);

		\draw [BrickRed,->] ($(0,-1.5) + (120:.2)$) arc (120:420:.2);
		\node [BrickRed] at (0,-2) {\small$\ZZ/6\ZZ$};

		\draw [BurntOrange,->] ($(30:1.8) + (-120:.2)$) arc (-120:180:.2);
		\node [BurntOrange,thick] at (25:2.3) [above] {\small$\ZZ/4\ZZ$};

		\node [white] at (90:2) {$\bullet$};
		\node [white] at (-30:2) {$\bullet$};
		\node [white] at (-150:2) {$\bullet$};
		\node at (90:2) {$\circ$};
		\node at (-30:2) {$\circ$};
		\node at (-150:2) {$\circ$};

		\begin{scope}[scale=.2]
			\draw [thick,BrickRed] (2,0) to[out=0,in=-90] (2.5,.5) to[out=90,in=0] (2,1) to[out=180,in=0] (0,0) -- (-2,0);

			\draw [line width=5pt,white] ($(0,0) + (30:2cm and 1.4cm)$) arc (30:60:2cm and 1.4cm);

			\draw [BrickRed,thick] (0,0) ellipse (2cm and 1.4cm);
			\node [BrickRed] at (2,0) {\scalebox{.5}{$\bullet$}};
			\node [BrickRed] at (-2,0) {\scalebox{.5}{$\bullet$}};	
		\end{scope}
		\begin{scope}[scale=.15,xshift=5cm,yshift=4.5cm]
			\draw [thick,BurntOrange] (0,0) ellipse (2cm and 1.4cm);

			\draw [line width=5pt,white] ($(1.4,-1.4) + (0:1.4cm and 2cm)$) arc (0:120:1.4cm and 2cm);

			\draw [thick,BurntOrange] (1.4,-1.4) ellipse (1.4cm and 2cm);
			\node [BurntOrange] at (0,-1.4) {\scalebox{.5}{$\bullet$}};
		\end{scope}

		\node at (2.5,0) {$\cong$};

		\begin{scope}[xshift=3.5cm]
			\draw [decoration={markings,mark=at position 0.5 with {\arrow{>}}},postaction={decorate}] (0,0) -- (90:2);
			\draw [decoration={markings,mark=at position 0.5 with {\arrow{>}}},postaction={decorate}] (0,0) -- (-30:2);
			\draw [decoration={markings,mark=at position 0.5 with {\arrow{>>}}},postaction={decorate}] (30:1) -- (90:2);
			\draw [decoration={markings,mark=at position 0.5 with {\arrow{>>}}},postaction={decorate}] (30:1) -- (-30:2);

			\node [white] at (90:2) {$\bullet$};
			\node [white] at (-30:2) {$\bullet$};
			\node at (90:2) {$\circ$};
			\node at (-30:2) {$\circ$};

			\node [BrickRed] at (0,0) {$\bullet$};
			\node [BurntOrange] at (30:1) {$\bullet$};
		\end{scope}
	\end{tikzpicture}
	\subcaption{
		The combinatorial moduli space of type $(1,1)$. It has two orbicells, corresponding to the two ribbon graphs of type $(1,1)$. Every point has stabiliser $\ZZ/2\ZZ$ given by the elliptic involution, except for the $3$-valent metric ribbon graph with all edge lengths being equal (with $\ZZ/6\ZZ$-stabiliser), and the $4$-valent metric ribbon graph with all edge lengths being equal (with $\ZZ/4\ZZ$-stabiliser).
	}
	\label{fig:comb:mod:torus}
	\end{subfigure}
	\caption{The combinatorial Teichmüller space of a one-holed torus (left), and the corresponding moduli space.}
	\label{fig:comb:torus}
\end{figure}

\subsection{The symplectic structure, the length function, and the integration formula}
\label{subsec:length}
As showed by Kontsevich \cite{Kon92}, the moduli space $\cM^{\textup{comb}}_{g,n}(L)$ carries a natural symplectic form $\omega$ that we call Kontsevich form. The associated cohomology class has deep connections with the moduli space of Riemann surfaces, as recalled in \Cref{subsec:int:numbers}.

Denote by $dG \coloneqq \frac{\omega^{3g-3+n}}{(3g-3+n)!}$ the volume form associated to $\omega$. The symplectic volumes
\begin{equation}
	V_{g,n}(L)
	\coloneqq
	\int_{\cM^{\textup{comb}}_{g,n}(L)} dG
\end{equation}
are finite, and they have been computed recursively by Kontsevich using matrix model techniques. It is worth mentioning that the symplectic volume form is proportional to the Lebesgue measure defined on each top-dimensional cell forming the combinatorial moduli space, and as such the volumes coincide with the asymptotic counting of metric ribbon graphs with edge lengths in $\frac{1}{k} \ZZ$ as $k \to \infty$. A geometric proof of Kontsevich's recursion, based on the existence of Fenchel--Nielsen coordinates and a Mirzakhani-type recursion for the constant function $1$ on the combinatorial Teichmüller space, can be found in \cite{ABCGLW}.

Let us recall the notion of combinatorial Fenchel--Nielsen coordinates. Fix an embedded metric ribbon graph $\GG \in \cT^{\textup{comb}}_{\Sigma}(L)$ and a free homotopy class $\gamma$ of a (non-null) simple closed curve in $\Sigma$. Consider the unique representative of $\gamma$ that has been homotoped to the embedded graph as a non-backtracking edge-path. We will refer to it as the geodesic representative. The geodesic length $\ell_{\GG}(\gamma)$ is defined by adding up the lengths of the edges visited by its geodesic representative. Thus, every free homotopy class $\gamma$ of non-trivial simple closed curves on $\Sigma$ defines a function $\ell(\gamma) \colon \cT^{\textup{comb}}_{\Sigma}(L) \to \RR_{>0}$ that assigns to $\GG$ the geodesic length $\ell_{\GG}(\gamma)$.

Consider now a pants decomposition $\mathcal{P}$ of $\Sigma$, that is a collection $(\gamma_m)_{m = 1}^{3g-3+n}$ of simple closed curves that cut $\Sigma$ into a disjoint union of pairs of pants. For a given embedded metric ribbon graph $\GG \in \cT^{\textup{comb}}_{\Sigma}(L)$, we can assign to each curve in $\mathcal{P}$ the data of two real numbers in $\RR_{>0} \times \RR$: the length $\ell_{\GG}(\gamma_m)$ and the twist $\tau_{\GG}(\gamma_m)$ of the gluing. Thus, we get a map
\begin{equation}
	\cT^{\textup{comb}}_{\Sigma}(L) \longrightarrow (\RR_{>0} \times \RR)^{\mathcal{P}},
	\qquad
	\GG \longmapsto \bigl( \ell_{\GG}(\gamma_m), \tau_{\GG}(\gamma_m) \bigr)_{m=1}^{3g-3+n} 
\end{equation}
called the combinatorial Fenchel--Nielsen coordinates. They are the analogue of Fenchel--Nielsen coordinates in hyperbolic geometry, and Dehn--Thurston coordinates in the theory of measured foliations. Lengths and twists form a global coordinate system on the combinatorial Teichmüller space that is Darboux for the Kontsevich symplectic form (canonically lifted to a mapping class group invariant form on the moduli space). This is the combinatorial analogue of Wolpert's magic formula \cite{Wol85} in the hyperbolic setting.

\begin{theorem}[Combinatorial Wolpert's formula {\cite{ABCGLW}}]
	For every pants decomposition, the Kontsevich form on $\cT^{\textup{comb}}_{\Sigma}(L)$ is canonically given by
	\begin{equation}
		\omega = \sum_{m=1}^{3g-3+n} d\ell_m \wedge d\tau_m .
	\end{equation}
\end{theorem}

The above formula allows for the integration of natural geometric functions defined on the combinatorial moduli space. This fact, which is the combinatorial analogue of Mirzakhani's integration formula \cite{Mir07a}, is one of the main ingredients in the geometric proof of the volume recursion. In order to state the formula, let us introduce some notation.

A stable graph consists of the data $\Gamma = ( \ms{V}(\Gamma), \ms{H}(\Gamma), (g_v)_{v \in \ms{V}(\Gamma)}, \nu, \iota )$ satisfying the following properties.
\begin{enumerate}
	\item
	$\ms{V}(\Gamma)$ is the set of vertices, equipped with the assignment of non-negative integers $(g_v)_{v \in \ms{V}(\Gamma)}$ called the genus decoration.
	
	\item
	$\ms{H}(\Gamma)$ is the set of half-edges, the map $\nu \colon \ms{H}(\Gamma) \to \ms{V}(\Gamma)$ associates to each half-edge the vertex it is incident to, and $\iota \colon \ms{H}(\Gamma) \to \ms{H}(\Gamma)$ is an involution that pairs half-edges together.

	\item
	The set of $2$-cycles of $\iota$ is the set of edges, denoted $\ms{E}(\Gamma)$ (self-loops are permitted).
	
	\item
	The set of $1$-cycles (i.e.\ fixed points) of $\iota$ is the set of leaves, denoted $\ms{\Lambda}(\Gamma)$. We require that leaves are labelled: there is a bijection $\ms{\Lambda}(\Gamma) \to \{1,\dots,n\}$, where $n \coloneqq |\ms{\Lambda}(\Gamma)|$.
	
	\item
	The pair $( \ms{V}(\Gamma), \ms{E}(\Gamma) )$ defines a connected graph.
	
	\item
	If $v$ is a vertex, denote by $n_v \coloneqq |\nu^{-1}(v)|$ its valency. We require that for each vertex $v$, the stability condition $2g_v - 2 + n_v > 0$ holds.
\end{enumerate}
For a given stable graph $\Gamma$, define its genus as
\begin{equation}
	g(\Gamma) \coloneqq \sum_{v \in \ms{V}(\Gamma)} g(v) + h^1(\Gamma),
\end{equation}
where $h^1(\Gamma)$ is the first Betti number of $\Gamma$. We denote by $\ms{G}_{g,n}$ the set of stable graphs of genus $g$ with $n$ leaves. An automorphism of $\Gamma$ consists of bijections of the sets $\ms{V}(\Gamma)$ and $\ms{H}(\Gamma)$ which leave invariant the structures $(g)_{v \in \ms{V}(\Gamma)}$, $\nu$, $\iota$, and the leaves labelling. We denote by $\Aut(\Gamma)$ the automorphism group of $\Gamma$.

Stable graphs naturally appear as mapping class group orbits of primitive multicurves. Let $\gamma = (\gamma_1,\dots,\gamma_r)$ be an ordered primitive multicurve, that is an $r$-tuple of free homotopic classes of simple closed curves on $\Sigma$ that are non-null, non-peripheral, and distinct. Denote by $\Gamma$ the mapping class group orbit $\MCG_{\Sigma} \cdot \gamma$. Then $\Gamma$ is identified with a stable graph with $r$ labelled edges, where the vertices and the genus decoration correspond to the connected components of $\Sigma \setminus \gamma$, the $s$-th edge corresponds to the curve $\gamma_s$, and the $i$-th leaf corresponds to the boundary component $\partial_i \Sigma$ (see \Cref{fig:stable:graph} for an example). Notice that, in contrast to the previous definition, stable graphs have now labelled edges.

For a given function $F \colon \RR_{>0}^n \times \RR_{>0}^r \to \RR$, define $F_{\Gamma} \colon \cT^{\textup{comb}}_{\Sigma}(L) \to \RR$ as
\begin{equation}
	F_{\Gamma}(\GG)
	\coloneqq
	 \sum_{ \alpha \in \Gamma }
		F\bigl( L,\ell_{\GG}(\alpha) \bigr) ,
\end{equation}
where $\ell_{\GG}(\alpha) = (\ell_{\GG}(\alpha_1),\dots,\ell_{\GG}(\alpha_r))$. The function $F_{\Gamma}$ descends naturally to a function on the moduli space $\cM^{\textup{comb}}_{g,n}(L)$, that we denote with the same symbol. Its integral is given by the following formula.

\begin{figure}
	\centering
	\begin{tikzpicture}[x=1pt,y=1pt,scale=.7]
		\draw[ForestGreen, thick] (123.0986, 715.8997) .. controls (133.4315, 711.7587) and (148.7158, 727.8794) .. (155.4745, 747.9229) .. controls (162.2332, 767.9664) and (160.4664, 791.9328) .. (142.596, 791.6668);
		\draw[ForestGreen, thick, opacity=.5] (142.596, 791.6668) .. controls (128.7207, 791.8959) and (122.3604, 773.948) .. (118.6818, 755.1069) .. controls (115.0033, 736.2659) and (114.0066, 716.5318) .. (123.0986, 715.8997);
		\draw[ForestGreen, thick] (125.2051, 647.9859) .. controls (133.4517, 648.2265) and (136.7259, 658.1132) .. (136.5737, 666.7475) .. controls (136.4216, 675.3817) and (132.8431, 682.7634) .. (125.9585, 683.178);
		\draw[ForestGreen, thick, opacity=.5] (125.9585, 683.178) .. controls (119.2801, 684.365) and (115.6401, 676.1825) .. (115.1338, 666.9978) .. controls (114.6275, 657.8131) and (117.255, 647.6263) .. (125.2051, 647.9859);
		\draw[ForestGreen, thick] (196.7013, 747.6512) .. controls (202.3132, 746.9713) and (205.1566, 761.4856) .. (204.0392, 773.3755) .. controls (202.9219, 785.2654) and (197.8438, 794.5309) .. (189.7564, 793.7345);
		\draw[ForestGreen, thick, opacity=.5] (189.7564, 793.7345) .. controls (183.5873, 793.214) and (181.7937, 782.607) .. (183.428, 770.7971) .. controls (185.0623, 758.9872) and (190.1245, 745.9743) .. (196.7013, 747.6512);

		\draw [thick,fill,fill opacity=.05] (34.6667, 784) .. controls (37.3333, 789.3333) and (42.6667, 794.6667) .. (48, 797.3333) .. controls (53.3333, 800) and (58.6667, 800) .. (61.3333, 797.3333) .. controls (64, 794.6667) and (64, 789.3333) .. (61.3333, 784) .. controls (58.6667, 778.6667) and (53.3333, 773.3333) .. (48, 770.6667) .. controls (42.6667, 768) and (37.3333, 768) .. (34.6667, 770.6667) .. controls (32, 773.3333) and (32, 778.6667) .. cycle;
		\draw [thick,fill,fill opacity=.05] (226.6667, 784) .. controls (229.3333, 778.6667) and (234.6667, 773.3333) .. (240, 770.6667) .. controls (245.3333, 768) and (250.6667, 768) .. (253.3333, 770.6667) .. controls (256, 773.3333) and (256, 778.6667) .. (253.3333, 784) .. controls (250.6667, 789.3333) and (245.3333, 794.6667) .. (240, 797.3333) .. controls (234.6667, 800) and (229.3333, 800) .. (226.6667, 797.3333) .. controls (224, 794.6667) and (224, 789.3333) .. cycle;
		\draw [thick,fill,fill opacity=.05] (34.6667, 656) .. controls (37.3333, 650.6667) and (42.6667, 645.3333) .. (48, 642.6667) .. controls (53.3333, 640) and (58.6667, 640) .. (61.3333, 642.6667) .. controls (64, 645.3333) and (64, 650.6667) .. (61.3333, 656) .. controls (58.6667, 661.3333) and (53.3333, 666.6667) .. (48, 669.3333) .. controls (42.6667, 672) and (37.3333, 672) .. (34.6667, 669.3333) .. controls (32, 666.6667) and (32, 661.3333) .. cycle;
		\draw [thick,fill,fill opacity=.05] (240, 642.6667) .. controls (245.3333, 645.3333) and (250.6667, 650.6667) .. (253.3333, 656) .. controls (256, 661.3333) and (256, 666.6667) .. (253.3333, 669.3333) .. controls (250.6667, 672) and (245.3333, 672) .. (240, 669.3333) .. controls (234.6667, 666.6667) and (229.3333, 661.3333) .. (226.6667, 656) .. controls (224, 650.6667) and (224, 645.3333) .. (226.6667, 642.6667) .. controls (229.3333, 640) and (234.6667, 640) .. cycle;

		\draw [thick] (55.9804, 799.3293) .. controls (114.6601, 789.1098) and (173.3518, 789.11) .. (232.0555, 799.3301);
		\draw [thick] (56.1994, 640.6785) .. controls (114.7331, 650.8928) and (173.3342, 650.8904) .. (232.0025, 640.6712);
		\draw [thick] (254.7771, 667.0311) .. controls (244.9257, 702.3437) and (244.9417, 737.7005) .. (254.8251, 773.1016);
		\draw [thick] (33.1768, 666.9038) .. controls (48.3923, 702.3013) and (48.3724, 737.7574) .. (33.1171, 773.2723);
		\draw [thick] (72, 760) .. controls (80, 740) and (112, 756) .. (104, 780);
		\draw [thick] (184, 752) .. controls (184, 724) and (204, 724) .. (216, 740);
		\draw [thick] (84.0585, 750.9585) .. controls (76, 768) and (92, 776) .. (105.2047, 770.7132);
		\draw [thick] (184.9058, 742.6833) .. controls (196, 752) and (212, 748) .. (209.224, 733.2131);
		\draw [thick] (108, 728) .. controls (88, 708) and (94, 694) .. (111, 687) .. controls (128, 680) and (156, 680) .. (160, 700);
		\draw [thick] (96.9318, 712.4965) .. controls (116, 720) and (128, 716) .. (136, 710) .. controls (144, 704) and (148, 696) .. (146.2236, 684.8747);

		\node at (38, 784) [above left] {$\partial_1 \Sigma$};
		\node at (250, 784) [above right] {$\partial_2 \Sigma$};
		\node at (250, 656) [below right] {$\partial_3 \Sigma$};
		\node at (38, 656) [below left] {$\partial_4 \Sigma$};
		\node [ForestGreen] at (142.596, 791.667) [above] {$\gamma_1$};
		\node [ForestGreen] at (125.205, 647.986) [below] {$\gamma_2$};
		\node [ForestGreen] at (189.756, 793.735) [above] {$\gamma_3$};

		\begin{scope}[xshift=2cm]
			\draw [thick] (352, 720) -- (323, 739);
			\draw [thick] (352, 720) -- (323, 701);
			\draw [thick] (464, 720) -- (493, 739);
			\draw [thick] (464, 720) -- (493, 701);

			\draw [thick,ForestGreen] (352, 720) .. controls (389.3333, 740.6667) and (426.6667, 740.6667) .. (464, 720);
			\draw [thick,ForestGreen] (352, 720) .. controls (389.3333, 699.3333) and (426.6667, 699.3333) .. (464, 720);
			\draw [thick,ForestGreen] (464, 720) .. controls (448, 736) and (456, 754) .. (464, 754) .. controls (472, 754) and (480, 736) .. (464, 720);
			
			\filldraw [thick,fill=white] (352, 720) circle[radius=10];
			\filldraw [thick,fill=white] (464, 720) circle[radius=10];

			\node at (323, 739) [above left] {$1$};
			\node at (493, 739) [above right] {$2$};
			\node at (493, 701) [below right] {$3$};
			\node at (323, 701) [below left] {$4$};
			\node [ForestGreen] at (408, 738) [above] {$\gamma_1$};
			\node [ForestGreen] at (408, 702) [below] {$\gamma_2$};
			\node [ForestGreen] at (464, 754) [above] {$\gamma_3$};
			\node at (352, 720) {$1$};
			\node at (464, 720) {$0$};
		\end{scope}
	\end{tikzpicture}
	\caption{The stable graph corresponding to an ordered tuple of curves.}
	\label{fig:stable:graph}
\end{figure}
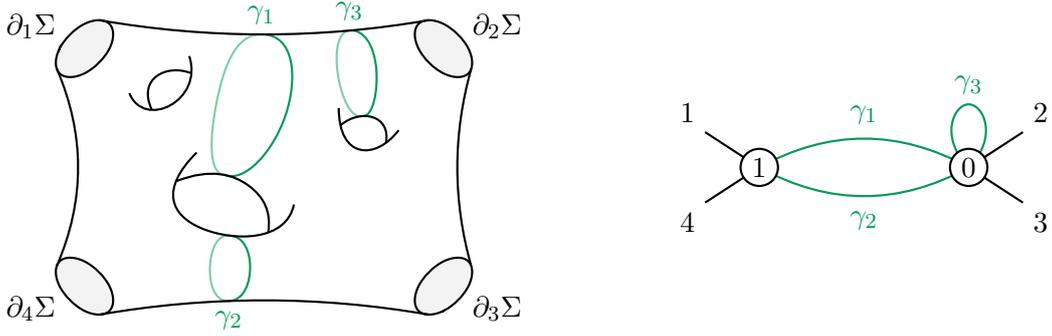

\begin{theorem}[Integration formula {\cite{ABCGLW}}] \label{thm:int:formula}
	Under suitable integrability conditions, the integral of $F_{\Gamma}$ over the combinatorial moduli space is given by
	\begin{equation} \label{eq:intFor}
		\int_{\cM^{\textup{comb}}_{g,n}(L)} F_{\Gamma} \, dG
		=
		\frac{1}{|\Aut(\Gamma)|} \int_{\RR_{>0}^r} F(L,\ell) \, V_{\Gamma}(L,\ell) \prod_{s=1}^r \ell_s \, d\ell_s ,
	\end{equation}
	where $V_{\Gamma}(L,\ell) \coloneqq \prod_{v \in \ms{V}(\Gamma)} V_{g_v,n_v}(L_v, \ell_v)$. Here $L_v$ (resp. $\ell_v$) denotes the tuple of lengths associated to the leaves (resp. half-edges that are not leaves) attached to $v$.
\end{theorem}

\subsection{Connection with intersection numbers and Aggarwal's asymptotic formula}
\label{subsec:int:numbers}
A notable aspect of the combinatorial moduli space $\cM^{\textup{comb}}_{g,n}(L)$ is its homeomorphism to the moduli space of Riemann surfaces $\cM_{g,n}$ (both considered as real topological orbifolds). This fact was proved by Harer \cite{Har86} using meromorphic differentials (based on the seminal works of Jenkins and Strebel and unpublished works of Thurston and Mumford) and by Penner and Bowditch--Epstein \cite{Pen87,BE88} using hyperbolic geometry. For instance, the moduli space $\cM^{\textup{comb}}_{1,1}(L)$ in \Cref{fig:comb:mod:torus} is nothing but the moduli of elliptic curves $\cM_{1,1} = [\SL(2,\ZZ) \backslash \mathbb{H}^2]$.

Consequently, any topological invariant of $\cM^{\textup{comb}}_{g,n}(L)$ can be translated into a corresponding topological invariant of $\cM_{g,n}$ and vice versa. As mentioned in the introduction, Harer and Zagier \cite{HZ86} utilise this fact to compute the Euler characteristic of the moduli space of curves, and Kontsevich employs it to prove Witten's conjecture \cite{Wit91,Kon92}. The former is a consequence of the following result due to Kontsevich.

\begin{theorem}[{Symplectic volumes as intersection numbers \cite{Kon92}}] \label{thm:Kont}
	The symplectic volumes satisfy
	\begin{equation}
		V_{g,n}(L)
		=
		\sum_{\substack{(d_1,\dots,d_n) \in \ZZ_{\geq 0}^k \\ |d| = 3g-3+n}}
			\braket{\tau_{d_1} \cdots \tau_{d_n}}_{g,n}
			\prod_{i=1}^n \frac{L_i^{2d_i}}{2^{d_i} d_i!} ,
	\end{equation}
	where
	\begin{equation}
		\langle \tau_{d_1} \cdots \tau_{d_n} \rangle_{g,n}
		\coloneqq
		\int_{\overline{\cM}_{g,n}} \psi_1^{d_1} \cdots \psi_n^{d_n}
	\end{equation}
	are the Witten--Kontsevich intersection numbers over the Deligne--Mumford compactification of the moduli space of Riemann surfaces. Here $\psi_i \coloneqq c_1(\mathcal{L}_i) \in H^2(\overline{\cM}_{g,n},\QQ)$ is the first Chern class of the $i$-th tautological line bundle $\mathcal{L}_i$, i.e.\ the line bundle whose fibre over $[C,p_1,\dots,p_n] \in \overline{\cM}_{g,n}$ is the cotangent line at $p_i$ of $C$.
\end{theorem}

The main tool to estimate volumes in the large genus limit is an asymptotic formula for the Witten--Kontsevich intersection numbers. The formula was conjectured by Delecroix--Goujard--Zograf--Zorich based on numerical data \cite{DGZZ21}, and proved shortly after by Aggarwal \cite{Agg21} through a combinatorial analysis of the associated Virasoro constraints. An alternative proof was recently given in \cite{GY22} and in \cite{EGGGL} through a combinatorial and resurgent analysis of determinantal formulae respectively.

\begin{theorem}[{Large genus asymptotic of intersection numbers \cite{Agg21}}] \label{thm:Agg}
    For any stable $(g,n)$ with $n^2 < g/800$ and $(d_1, \dots, d_n) \in \ZZ_{\geq 0}^{n}$ with $|d| = 3g-3+n$, we have as $g \to \infty$
    \begin{equation} \label{eq:psi:sim}
        \braket{\!\braket{ \tau_{d_1} \cdots \tau_{d_n} }\!}_{g,n}
        \coloneqq
        \braket{ \tau_{d_1} \cdots \tau_{d_n} }_{g,n} \prod_{i=1}^n (2d_i+1)!!
        =
        \frac{(6g-5+2n)!!}{g! \, 24^g} \bigl( 1 + \lo(1) \bigr)
    \end{equation}
    uniformly in $n$ and $(d_1, \dots, d_n)$.
\end{theorem}

The large genus asymptotics for the symplectic volumes read
\begin{equation} \label{eq:V}
\begin{split}
	V_{g,n}(L)
	\sim
	\frac{(6g-5+2n)!!}{g! \, 24^g} \, [z^{6g-6+3n}] \prod_{i=1}^n \frac{\sinh(L_i z)}{L_i} ,
\end{split}
\end{equation}
where $[z^d] \, f(z)$ denotes the coefficient of $z^d$ in the Taylor expansion of $f(z)$ around $z = 0$. We will come back to this observation in the next section, when discussing estimates on Kontsevich volumes. 

The asymptotic formula \labelcref{eq:psi:sim} becomes false if $n$ grows too rapidly compared to $g$ (for instance, consider the exact formula $\langle \tau_{d_1} \cdots \tau_{d_n} \rangle_{0,n} = \binom{n-3}{d_1,\ldots,d_n}$). However, the following estimate always holds.

\begin{theorem}[{Uniform bound on intersection numbers \cite{Agg21}}] \label{thm:uniBd}
	For any stable $(g,n)$ and $(d_1, \dots, d_n) \in \ZZ_{\geq 0}^{n}$ with $|d| = 3g-3+n$, we have
	\begin{equation}
		\braket{\!\braket{ \tau_{d_1} \cdots \tau_{d_n} }\!}_{g,n}
		\leq
		\left( \frac{3}{2} \right)^{n-1} \frac{(6g-5+2n)!!}{g! \, 24^g} .
	\end{equation}
\end{theorem}

\subsection{Random metric ribbon graphs and the method of moments}
\label{subsec:random:MRG}
As the combinatorial moduli space has a finite volume, we can turn it into a probability space. More precisely, given a measurable set $A \subseteq \cM_{g,n}^{\comb}(L)$, define
\begin{equation}
	\PP[A] \coloneqq \frac{1}{V_{g,n}(L)} \int_{A} dG.
\end{equation}
Given random variables $X_1,\dots,X_p \colon \cM_{g,n}^{\comb}(L) \to \RR$ and sets $E_1,\dots,E_p \subset \RR$, define
\begin{equation}
	\PP[X_1 \in E_1,\dots,X_p \in E_p] \coloneqq \PP\left[ \bigcap_{i=1}^p \Set{ G \in \cM_{g,n}^{\comb}(L) | X_i(G) \in E_i } \right].
\end{equation}
If $E_i = \{ x_i \}$ is a singleton, we simply denote the above quantity as $\PP[X_1 = x_1,\dots,X_p=x_p]$. For $p =1$, we refer to the function $\PP[X=x]$ as the probability mass function of $X$. We also define the expectation value of a random variable $X$ to be
\begin{equation}
	\EE[X] \coloneqq \frac{1}{V_{g,n}(L)} \int_{\cM_{g,n}^{\comb}(L)} X(G) \, dG.
\end{equation}
Often, we write $\EE[X(\bm{G}_{g,L})]$ where $\bm{G}_{g,L}$ is a random element in $\cM_{g,n}^{\comb}(L)$ to emphasise the dependence on the genus and the number/lengths of the boundary components.

As in \cite{MP19,JL23}, the main probabilistic tool used in this paper is the method of moments, which we now recall. Let $(\Omega, \cB, \PP)$ be a probability space, and let $X \colon \Omega \to \ZZ_{\geq 0}$ be an integer-valued random variable. For any integer $r$, define
\begin{equation}
	(X)_r \coloneqq X (X-1)\cdots (X-r+1).
\end{equation}
The expectation value $\EE[(X)_r]$ is called the $r$-th factorial moment of $X$.

A special class of integer-valued random variables is constituted by those following a Poisson distribution. We say that $X$ is Poisson distributed with mean $\lambda \in \RR_{>0}$ if its probability mass function satisfies
\begin{equation}\label{eq:prob:Poisson}
	\PP[X = k] = \re^{-\lambda} \, \frac{\lambda^k}{k!}
	\qquad
	\text{for all } k \in \ZZ_{\geq 0}.
\end{equation}
In this case, it can be shown that the $r$-th factorial moment of $X$ is given by $\EE[(X)_r] = \lambda^r$ for all integers $r \in \ZZ_{\geq 0}$. The method of moments asserts that the converse is also true. In other words, an integer-valued random variable is Poisson distributed with mean $\lambda$ if and only if its its $r$-th factorial moment is given by $\lambda^r$. More precisely, we have the following result (see for instance \cite{Bol01}).

\begin{theorem}[Method of moments] \label{thm:meth:mom}
	Let $\{ (\Omega_d, \cB_d, \PP_d) \}_{d \in \ZZ_{\geq 1}}$ be a sequence of probability spaces, and let $(X_{d,1}, \dots, X_{d,p}) \colon \Omega_d \to \ZZ_{\geq 0}^p$ be an integer-valued random vector. Suppose there exists $(\lambda_1,\dots,\lambda_p) \in \RR_{>0}^p$ such that
	\begin{equation}
		\lim_{d \to \infty} \EE \left[ \prod_{i=1}^p (X_{d,i})_{r_i} \right]
		=
		\prod_{i=1}^p \lambda_i^{r_i}
	\end{equation}
	for all $(r_1, \dots, r_p) \in \ZZ_{\geq 0}^p$. Then $(X_{d,1}, \dots, X_{d,p})_{d \in \ZZ_{\geq 1}}$ converges jointly in distribution to a vector of independent Poisson variables with parameters $(\lambda_1, \dots, \lambda_p)$.
	In other words, for all $(k_1, \dots, k_p) \in \ZZ_{\geq 0}^p$:
	\begin{equation}
		\lim_{d \to \infty} \PP\bigl[ X_{d,1} = k_1, \dots, X_{d,p} = k_p \bigr]
		=
		\prod_{i=1}^p \re^{-\lambda_i} \, \frac{\lambda_i^{k_i}}{k_i!} .
	\end{equation}
\end{theorem}

Poisson distributions appear for instance in the context of point processes. A point process (on the positive real line) is a collection of integer-valued random variables $X_{[a,b)} \colon \Omega \to \ZZ_{\geq 0}$ indexed by positive real numbers $0 \le a < b$. The value $X_{[a,b)}$ can be thought of as the number of events happening in the interval $[a,b)$.

A point process is called Poisson if there exists a locally integrable non-negative function $\lambda \colon \RR_{\ge 0} \to \RR_{\geq 0}$, called the intensity, such that $X_{[a,b)}$ is Poisson distributed with mean
\begin{equation}
	\int_a^b \lambda(\ell) \, d\ell,
\end{equation}
and for any disjoint intervals $[a_1, b_1], \dots, [a_p, b_p] \subset \RR_{\geq 0}$, $X_{[a_1, b_1)}, \dots, X_{[a_p, b_p]}$ are independent.

\section{The combinatorial length spectrum}
\label{sec:comb:spec}
\subsection{Setup}
Fix a topological surface $\Sigma$ of type $(g,n)$ and a cell in $\cT_{\Sigma}^{\comb}(L)$, i.e.\ an embedded ribbon graph $G \hookrightarrow \Sigma$ that is a retract. A free homotopy class of a (non-null, non-peripheral) closed curve $\mathtt{C}$ is called a cycle if its unique non-backtracking edge-path representative visits each edge at most once. Notice that the notion of cycle depends on the underlying embedded ribbon graph. See \Cref{fig:cycle} for an example and a non-example.

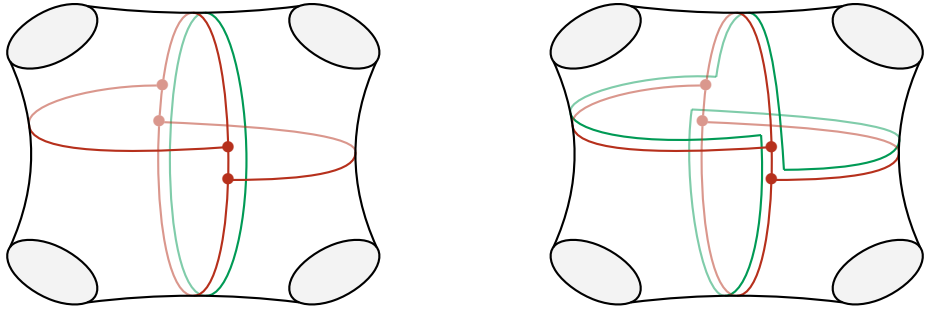
\begin{figure}
	\centering
	\begin{tikzpicture}[x=1pt,y=1pt,scale=.55]

		\draw [ForestGreen, thick] (200.4821, 518.9001) .. controls (216.4226, 518.4413) and (228.2113, 563.2207) .. (228.1876, 611.9957) .. controls (228.1639, 660.7707) and (216.3278, 713.5414) .. (199.8243, 713.0728);
		\draw [ForestGreen, thick, opacity=.5] (200.4821, 518.9001) .. controls (184.2458, 518.9054) and (176.1229, 563.4527) .. (176.0729, 611.9949) .. controls (176.023, 660.5371) and (184.0459, 713.0742) .. (199.8243, 713.0728);

		\draw [BrickRed, thick] (192.122, 518.961) .. controls (207.846, 518.744) and (215.923, 563.372) .. (215.8808, 611.995) .. controls (215.8385, 660.618) and (207.677, 713.236) .. (191.759, 713.019);
		\draw [BrickRed, thick, opacity=.5] (191.752, 713.019) .. controls (175.943, 713.246) and (167.9715, 660.623) .. (168.0973, 611.9973) .. controls (168.223, 563.3715) and (176.446, 518.743) .. (192.122, 518.961);
		\draw [BrickRed, thick] (215.676, 598.365) .. controls (274.222, 598.215) and (303.1163, 604.59) .. (302.359, 617.49);
		\draw [BrickRed, thick, opacity=.5] (302.359, 617.49) .. controls (303.0103, 627.6507) and (258.492, 634.5257) .. (168.804, 638.115);
		\draw [BrickRed, thick] (215.782, 620.924) .. controls (126.2979, 614.536) and (81.1069, 619.603) .. (80.209, 636.125);
		\draw [BrickRed, thick, opacity=.5] (80.209, 636.125) .. controls (77.3095, 650.988) and (128, 664) .. (171.056, 663.093);

		\draw [thick, shift={(96, 515.852)}, yscale=0.7222, fill,fill opacity=.05 ] (0, 0) .. controls (-10.6667, 5.3333) and (-21.3333, 16) .. (-26.6667, 26.6667) .. controls (-32, 37.3333) and (-32, 48) .. (-26.6667, 53.3333) .. controls (-21.3333, 58.6667) and (-10.6667, 58.6667) .. (0, 53.3333) .. controls (10.6667, 48) and (21.3333, 37.3333) .. (26.6667, 26.6667) .. controls (32, 16) and (32, 5.3333) .. (26.6667, 0) .. controls (21.3333, -5.3333) and (10.6667, -5.3333) .. cycle;
		\draw [thick, shift={(261.333, 535.111)}, yscale=0.7222, fill,fill opacity=.05] (0, 0) .. controls (5.3333, 10.6667) and (16, 21.3333) .. (26.6667, 26.6667) .. controls (37.3333, 32) and (48, 32) .. (53.3333, 26.6667) .. controls (58.6667, 21.3333) and (58.6667, 10.6667) .. (53.3333, 0) .. controls (48, -10.6667) and (37.3333, -21.3333) .. (26.6667, -26.6667) .. controls (16, -32) and (5.3333, -32) .. (0, -26.6667) .. controls (-5.3333, -21.3333) and (-5.3333, -10.6667) .. cycle;
		\draw [thick, shift={(96, 677.629)}, yscale=0.7222, fill,fill opacity=.05] (0, 0) .. controls (10.6667, 5.3333) and (21.3333, 16) .. (26.6667, 26.6667) .. controls (32, 37.3333) and (32, 48) .. (26.6667, 53.3333) .. controls (21.3333, 58.6667) and (10.6667, 58.6667) .. (0, 53.3333) .. controls (-10.6667, 48) and (-21.3333, 37.3333) .. (-26.6667, 26.6667) .. controls (-32, 16) and (-32, 5.3333) .. (-26.6667, 0) .. controls (-21.3333, -5.3333) and (-10.6667, -5.3333) .. cycle;
		\draw [thick, shift={(314.667, 696.888)}, yscale=0.7222, fill,fill opacity=.05] (0, 0) .. controls (-5.3333, 10.6667) and (-16, 21.3333) .. (-26.6667, 26.6667) .. controls (-37.3333, 32) and (-48, 32) .. (-53.3333, 26.6667) .. controls (-58.6667, 21.3333) and (-58.6667, 10.6667) .. (-53.3333, 0) .. controls (-48, -10.6667) and (-37.3333, -21.3333) .. (-26.6667, -26.6667) .. controls (-16, -32) and (-5.3333, -32) .. (0, -26.6667) .. controls (5.3333, -21.3333) and (5.3333, -10.6667) .. cycle;
		\draw [thick, shift={(67.223, 552.295)}, yscale=0.7222] (0, 0) .. controls (19.1849, 58.8053) and (19.2167, 117.5507) .. (0.0954, 176.236);
		\draw [thick, shift={(316.65, 679.536)}, yscale=0.7222] (0, 0) .. controls (-19.1, -58.6493) and (-19.0633, -117.365) .. (0.11, -176.147);
		\draw [thick, shift={(119.833, 514.342)}, yscale=0.7222] (0, 0) .. controls (48.1293, 8.5047) and (96.1987, 8.5277) .. (144.208, 0.069);
		\draw [thick, shift={(120, 717.591)}, yscale=0.7222] (0, 0) .. controls (47.998, -8.4433) and (96.0017, -8.4413) .. (144.011, 0.006);
		
		\node [white] at (171.056, 663.093) {$\bullet$};
		\node [white] at (168.804, 638.115) {$\bullet$};
		\node [BrickRed,opacity=.5] at (171.056, 663.093) {$\bullet$};
		\node [BrickRed,opacity=.5] at (168.804, 638.115) {$\bullet$};
		\node [BrickRed] at (215.676, 598.365) {$\bullet$};
		\node [BrickRed] at (215.782, 620.924) {$\bullet$};

		\begin{scope}[xshift=13cm]

			\draw [ForestGreen, thick, opacity=.5] (199.9308, 713.0743)  .. controls (189.3772, 713.0737) and (182.0971, 698.4117) .. (178.0905, 669.0882);
			\draw [ForestGreen, thick] (199.9308, 713.0743)  .. controls (211.7884, 713.4267) and (219.8583, 677.5075) .. (224.1407, 605.3169);
			\draw [ForestGreen, thick] (224.1407, 605.3169)  .. controls (276.5292, 605.678) and (302.7223, 612.6605) .. (302.72, 626.2643);
			\draw [ForestGreen, thick, opacity=.5] (302.72, 626.2643)  .. controls (303.785, 635.166) and (256.6768, 641.9779) .. (161.3953, 646.7001);
			\draw [ForestGreen, thick, opacity=.5] (178.0905, 669.0882)  .. controls (133.7051, 672.0363) and (75.4731, 657.5049) .. (78.8327, 644.1781);
			\draw [ForestGreen, thick] (78.8327, 644.1781)  .. controls (80.5834, 632.8758) and (120.2153, 619.397) .. (208.5921, 629.1344);
			\draw [ForestGreen, thick] (208.5921, 629.1344)  .. controls (212.5112, 561.0229) and (199.1187, 518.8961) .. (184.0373, 518.9025);
			\draw [ForestGreen, thick, opacity=.5] (184.0373, 518.9025)  .. controls (167.6022, 518.4275) and (154.8293, 598.7806) .. (161.3953, 646.7001);

			\draw [BrickRed, thick] (192.122, 518.961) .. controls (207.846, 518.744) and (215.923, 563.372) .. (215.8808, 611.995) .. controls (215.8385, 660.618) and (207.677, 713.236) .. (191.759, 713.019);
			\draw [BrickRed, thick, opacity=.5] (191.752, 713.019) .. controls (175.943, 713.246) and (167.9715, 660.623) .. (168.0973, 611.9973) .. controls (168.223, 563.3715) and (176.446, 518.743) .. (192.122, 518.961);
			\draw [BrickRed, thick] (215.676, 598.365) .. controls (274.222, 598.215) and (303.1163, 604.59) .. (302.359, 617.49);
			\draw [BrickRed, thick, opacity=.5] (302.359, 617.49) .. controls (303.0103, 627.6507) and (258.492, 634.5257) .. (168.804, 638.115);
			\draw [BrickRed, thick] (215.782, 620.924) .. controls (126.2979, 614.536) and (81.1069, 619.603) .. (80.209, 636.125);
			\draw [BrickRed, thick, opacity=.5] (80.209, 636.125) .. controls (77.3095, 650.988) and (128, 664) .. (171.056, 663.093);

			\draw [thick, shift={(96, 515.852)}, yscale=0.7222, fill,fill opacity=.05] (0, 0) .. controls (-10.6667, 5.3333) and (-21.3333, 16) .. (-26.6667, 26.6667) .. controls (-32, 37.3333) and (-32, 48) .. (-26.6667, 53.3333) .. controls (-21.3333, 58.6667) and (-10.6667, 58.6667) .. (0, 53.3333) .. controls (10.6667, 48) and (21.3333, 37.3333) .. (26.6667, 26.6667) .. controls (32, 16) and (32, 5.3333) .. (26.6667, 0) .. controls (21.3333, -5.3333) and (10.6667, -5.3333) .. cycle;
			\draw [thick, shift={(261.333, 535.111)}, yscale=0.7222, fill,fill opacity=.05] (0, 0) .. controls (5.3333, 10.6667) and (16, 21.3333) .. (26.6667, 26.6667) .. controls (37.3333, 32) and (48, 32) .. (53.3333, 26.6667) .. controls (58.6667, 21.3333) and (58.6667, 10.6667) .. (53.3333, 0) .. controls (48, -10.6667) and (37.3333, -21.3333) .. (26.6667, -26.6667) .. controls (16, -32) and (5.3333, -32) .. (0, -26.6667) .. controls (-5.3333, -21.3333) and (-5.3333, -10.6667) .. cycle;
			\draw [thick, shift={(96, 677.629)}, yscale=0.7222, fill,fill opacity=.05] (0, 0) .. controls (10.6667, 5.3333) and (21.3333, 16) .. (26.6667, 26.6667) .. controls (32, 37.3333) and (32, 48) .. (26.6667, 53.3333) .. controls (21.3333, 58.6667) and (10.6667, 58.6667) .. (0, 53.3333) .. controls (-10.6667, 48) and (-21.3333, 37.3333) .. (-26.6667, 26.6667) .. controls (-32, 16) and (-32, 5.3333) .. (-26.6667, 0) .. controls (-21.3333, -5.3333) and (-10.6667, -5.3333) .. cycle;
			\draw [thick, shift={(314.667, 696.888)}, yscale=0.7222, fill,fill opacity=.05] (0, 0) .. controls (-5.3333, 10.6667) and (-16, 21.3333) .. (-26.6667, 26.6667) .. controls (-37.3333, 32) and (-48, 32) .. (-53.3333, 26.6667) .. controls (-58.6667, 21.3333) and (-58.6667, 10.6667) .. (-53.3333, 0) .. controls (-48, -10.6667) and (-37.3333, -21.3333) .. (-26.6667, -26.6667) .. controls (-16, -32) and (-5.3333, -32) .. (0, -26.6667) .. controls (5.3333, -21.3333) and (5.3333, -10.6667) .. cycle;
			\draw [thick, shift={(67.223, 552.295)}, yscale=0.7222] (0, 0) .. controls (19.1849, 58.8053) and (19.2167, 117.5507) .. (0.0954, 176.236);
			\draw [thick, shift={(316.65, 679.536)}, yscale=0.7222] (0, 0) .. controls (-19.1, -58.6493) and (-19.0633, -117.365) .. (0.11, -176.147);
			\draw [thick, shift={(119.833, 514.342)}, yscale=0.7222] (0, 0) .. controls (48.1293, 8.5047) and (96.1987, 8.5277) .. (144.208, 0.069);
			\draw [thick, shift={(120, 717.591)}, yscale=0.7222] (0, 0) .. controls (47.998, -8.4433) and (96.0017, -8.4413) .. (144.011, 0.006);
			
			\node [white] at (171.056, 663.093) {$\bullet$};
			\node [white] at (168.804, 638.115) {$\bullet$};
			\node [BrickRed,opacity=.5] at (171.056, 663.093) {$\bullet$};
			\node [BrickRed,opacity=.5] at (168.804, 638.115) {$\bullet$};
			\node [BrickRed] at (215.676, 598.365) {$\bullet$};
			\node [BrickRed] at (215.782, 620.924) {$\bullet$};

		\end{scope}
	\end{tikzpicture}
	\caption{Example (left) and non-example (right) of a cycle on an embedded metric ribbon graph of type $(0,4)$.}
	\label{fig:cycle}
\end{figure}

The main object of study is the bottom part of the length spectrum of cycles on metric ribbon graphs, i.e.\ the function $N_{g,L,[a,b)} \colon \cT_{\Sigma}^{\comb}(L) \to \ZZ_{\geq 0}$ defined as
\begin{equation}
	N_{g,L,[a,b)}(\GG)
	\coloneqq
	\left\vert
	\Set{
		\substack{
			\displaystyle \text{$\mathtt{C}$ cycle in $\GG$} \\[.25em]
			\displaystyle \text{with length $\ell_{\GG}(\mathtt{C}) \in [a,b)$}
		}
	}
	\right\vert .
\end{equation}
The function $N_{g,L,[a,b)}$ is mapping class group invariant, so it descends to a function on the combinatorial moduli space (denoted with the same symbol) that we can regard as a point process on $\RR_{\ge 0}$. As presented in the introduction, our main result (\Cref{thm:PPP}) is the fact that $N_{g,L,[a,b)}$ converges in distribution as $g \to \infty$ to a Poisson point process of intensity given by
\begin{equation}
	\lambda_{\mu}(\ell) \coloneqq \frac{\cosh(\ell/\mu) - 1}{\ell},
\end{equation}
under the following assumption on the scaling of the boundary components.

\textit{Boundary scaling assumption.} The boundary lengths $L = (L_1,\dots,L_n)$ are positive real functions of $g$ and there exists $\mu > 0$ such that, as $g \to \infty$,
\begin{equation}\label{eq:bnd:scaling}
	|L| \coloneqq L_1 + \cdots + L_n \sim \mu 12 g.
\end{equation}

As outlined in the introduction, the boundary scaling assumption is naturally explained by the fact that, for a random ribbon graph of type $(g,n)$ with constant edge-lengths $\mu$, the total boundary length is given by $|L| = 2\mu (6g-6+3n) \sim \mu 12 g$. It is also worth mentioning that the mean $\int_{a}^{b} \lambda_{\mu}(\ell) d\ell$ can be written as $\frac{1}{2} \int_{a/\mu}^{b/\mu} \mathcal{S}^2(\ell) \, \ell \, d\ell$, where
\begin{equation}
	\mathcal{S}(\ell) \coloneqq \frac{\sinh(\ell/2)}{\ell/2} 
\end{equation}
is the ubiquitous function appearing in the enumerative theory of Riemann surfaces (see for instance \cite{OP06}). The square and the symmetry factor $\frac{1}{2}$ naturally appear by assigning the function $\mathcal{S}$ to the two unlabelled branches of the cycle $\mathtt{C}$.  Moreover, the measure $\ell \, d\ell$ is naturally interpreted as a length-twist measure, after integrating out the twist parameter.

Following \cite{MP19}, the strategy consists in employing the method of moments together with the key observation that the function
\begin{equation}
	N_{g,L,I} \coloneqq \prod_{i=1}^p \left( N_{g,L,[a_i,b_i)} \right)_{r_i}
\end{equation}
has a simple geometric interpretation: it counts the number of ordered $p$-tuples where the $i$-th item is the number of ordered $r_i$-tuples of cycles with length in $[a_i, b_i)$. Here we suppose that all intervals $[a_i,b_i)$ are disjoint, and we denoted $I \coloneqq \prod_{i=1}^p [a_i,b_i)^{r_i}$.

Another key step in the proof is to split $N_{g,L,I}$ as a sum of three terms:
\begin{equation} \label{eq:trinity}
	N_{g,L,I}
	=
	N_{g,L,I}^{\circ} + N_{g,L,I}^{\asymp} + N_{g,L,I}^{\times},
\end{equation}
where we have set:
\begin{itemize} \setlength\itemsep{.5em}
	\item $N_{g,L,I}^{\circ}$ for the number of tuples such that all cycles are simple, distinct, and non-separating,

	\item $N_{g,L,I}^{\asymp}$ for the number of tuples such that all cycles are simple, distinct, and separating,

	\item $N_{g,L,I}^{\times}$ for the number of tuples such that at least one cycle is non-simple, or at least one pair of cycles intersect.
\end{itemize}
We also denote by $\bar{N}_{g,L,I}^{\ast}$ the corresponding counting where ``cycle'' is replaced by ``closed curve''. Clearly, $N_{g,L,I}^{\ast} \le \bar{N}_{g,L,I}^{\ast}$ for $\ast \in \{ \circ, \asymp, \times \}$. The goal of \Cref{subsec:negl} is to prove the following claims: under the boundary scaling assumption \labelcref{eq:bnd:scaling}, as $g \to \infty$:
\begin{enumerate}[label=\textsc{\alph*}), ref=(\textsc{\alph*})] \setlength\itemsep{.5em}
	\item\label{itm:A}
		$\EE[ \bar{N}_{g,L,I}^{\circ} ] \sim \prod_{i=1}^p ( \int_{a_i}^{b_i} \lambda_\mu(x) \, dx )^{r_i}$,

	\item\label{itm:B}
	$\EE[ N_{g,L,I}^{\asymp} ] = \bO(g^{-1/2})$,

	\item\label{itm:C}
	$\EE[ N_{g,L,I}^{\times} ] = \bO(g^{-1/2})$,

	\item\label{itm:D}
	$\EE[ \bar{N}_{g,L,I}^{\circ} ] - \EE[ N_{g,L,I}^{\circ} ] = \bO(g^{-1/2})$.
\end{enumerate}
Claims \ref{itm:A}--\ref{itm:D}, together with the splitting \labelcref{eq:trinity}, imply that
\begin{equation}
	\lim_{g \to \infty} \EE\left[ N_{g,L,I} \right]
	=
	\prod_{i=1}^p \left( \int_{a_i}^{b_i} \lambda_\mu(x) \, dx \right)^{r_i} .
\end{equation}
Together with the method of moments, the main result of the paper follows.

The key ingredients that contribute to the proof of Claims~\ref{itm:A}--\ref{itm:D} are: an estimate on the Kontsevich volumes, the integration formula, and an estimate on a certain sum over stable graphs.

\subsection{Volume estimates}
As explained in \Cref{subsec:int:numbers}, the large genus asymptotics of the Kontsevich volumes are expressed in terms of coefficients of the form $[z^{6g-6+3n}] \prod_{i=1}^n \sinh(L_i z)$. More generally, in this section we are going to compute the $g \to \infty$ asymptotic behaviour of coefficients of the form
\begin{equation}
	\mathfrak{F}_{g}
	\coloneqq
	[z^{6g-6+3n}] \, F(z),
	\qquad \text{where} \qquad
	F(z)
	\coloneqq
	\Bigg( \prod_{i=1}^{n} \sinh(L_i z) \Bigg)
	\Bigg( \prod_{j=1}^{m} \sinh(\ell_j z) \Bigg),
\end{equation}
under the following assumptions:
\begin{itemize}
	\item
		$n \geq 1$ and $m \ge 0$ are fixed integers,

	\item
		$L = (L_1,\dots,L_n)$ and $\ell = (\ell_1,\dots,\ell_m)$ are positive real functions of $g$, with $|L| \sim \mu 12g$ (the boundary scaling assumption) and $|\ell| \leq K$ for some $K > 0$.
\end{itemize}
We omit the dependence of $\mathfrak{F}_{g}$ on $(L,\ell)$ for simplicity. In order to compute the asymptotic behaviour of $\mathfrak{F}_{g}$, we apply the saddle-point method (see \cite[Chapter~VIII]{FS09}).

\begin{proposition} \label{prop:sddl:pnt}
	Under the assumptions above, as $g \to \infty$ uniformly for $\ell$ in any compact set of $\RR_{\geq 0}^m$,
	\begin{equation} \label{eq:SgSim}
		\mathfrak{F}_{g}
		\sim
		\frac{\rho^{-(6g-6+3n)}}{\sqrt{3\pi g}}
		\Biggl( \prod_{i=1}^{n} \sinh(L_i \rho) \Biggr)
		\Bigg( \prod_{j=1}^{m} \sinh\left( \frac{\ell_j}{2\mu} \right) \Bigg)
	\end{equation}
	where $\rho \coloneqq (6g-6+3n)/|L| \sim 1/(2\mu)$.
\end{proposition}

\begin{figure}
	\begin{center}
	\begin{subfigure}[t]{.4\textwidth}
		\centering
		\includegraphics[width=\textwidth]{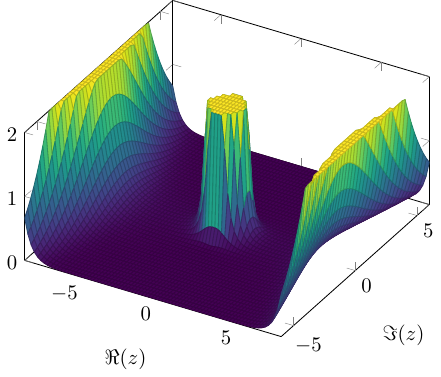}
	\end{subfigure}
	\hspace{1cm}
	\begin{subfigure}[t]{.4\textwidth}
		\centering
		\includegraphics[width=\textwidth]{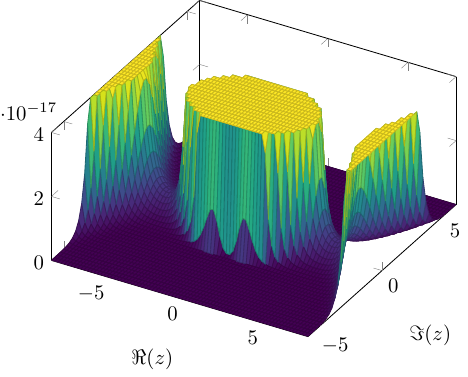}
	\end{subfigure}
	\caption{
		The graph of the absolute value of the function $F(z)z^{-(6g-5+3n)}$ for $g=2$ (left) and $g=10$ (right), with $n=m=1$, $L_1 = g$, and $\ell_1 = 1$. Notice the two saddle-points at $z \sim \pm 6$, becoming more and more pronounced as $g$ increases.
	}
	\label{fig:saddle}
	\end{center}
\end{figure}

\begin{proof}
	The function $F$ is holomorphic on the whole complex $z$-plane. By Cauchy's integral formula, for any radius $r \in \RR_{>0}$,
	\[
		\mathfrak{F}_{g}
		=
		\frac{1}{2 \pi \ri}
		\oint_{|z| = r} \frac{F(z)}{z^{6g-5+3n}} \, dz
		=
		\frac{r^{-(6g-6+3n)}}{\pi}
		\int_{-\pi/2}^{\pi/2}
		F(r \mkern1mu \re^{\ri \theta}) \, \re^{-(6g-6+3n) \ri \theta} \, d\theta.
	\]
	Here we used the parity property of the integrand and reduce the integral over a semicircle. See \Cref{fig:saddle} for a 3D plot of the absolute value of the integrand. Following the terminology of \cite[Chapter~VIII]{FS09}, in order to apply the saddle-point method, we shall choose:
	\begin{itemize}
		\item a radius so that the integration contour is an approximate saddle-point contour,

		\item a splitting of the approximate saddle-point contour as $\mathcal{C}_{\textup{centr}} + \mathcal{C}_{\textup{tail}}$, so that the following holds true.

		\begin{itemize}
			\item \textsc{Tail pruning.} The integral along $\mathcal{C}_{\textup{tail}}$ is negligible.

			\item \textsc{Central approximation.} The integral along $\mathcal{C}_{\textup{centr}}$ is well-approximated by an incomplete Gaussian integral.

			\item \textsc{Tail completion.} The incomplete Gaussian integral is asymptotically equivalent to the complete Gaussian integral.
		\end{itemize}
	\end{itemize}
	We first start with choosing the radius $r$ and the splitting.
	Let us write the integrand $F(r \mkern1mu \re^{\ri \theta}) \, \re^{-(6g-6+3n) \ri \theta}$ as $\re^{\mkern2mu f(\theta)}$, where
	\[
		f(\theta)
		\coloneqq
		\sum_{i=1}^{n} \log \sinh(L_i r \mkern1mu \re^{\ri \theta}) + \sum_{j=1}^{m} \log \sinh(\ell_j r \mkern1mu \re^{\ri \theta}) - (6g-6+3n) \ri \theta.
	\]
	In what follows, prime notation is used to denote derivatives with respect to $\theta$.
	The saddle-point equation for the radius $r$, namely
	\[
		f'(0)
		=
		\ri \sum_{i=1}^{n} L_i r \coth(L_i r)
		+
		\ri \sum_{j=1}^{m} \ell_j r \coth(\ell_j r)
		- (6g-6+3n) \ri
		= 0,
	\]
	has a unique solution, which is asymptotically equivalent to $1/(2\mu)$ as $g \to \infty$. We use $\rho = \rho(g) \coloneqq (6g-6+3n)/|L| \sim 1/(2\mu)$ as an approximate saddle-point contour. Near the saddle-point, we have $f'(0) = \bO(1)$ as $g \to \infty$. Moreover, we have
	\[
	\begin{aligned}
		f''(0)
		&=
		\sum_{i=1}^{n} \bigl(
			L_i^2 \rho^2 \csch^2 (L_i \rho) - L_i \rho \coth(L_i \rho)
		\bigr) +
		\sum_{j=1}^{m} \bigl(
			\ell_j^2 \rho^2 \csch^2 (\ell_j \rho) - \ell_j \rho \coth(\ell_i \rho)
		\bigr) \\
        &=
		- 6g + \bO(1), \\
		f'''(0)
		&=
		\ri \sum_{i=1}^{n} \bigl(
			3 L_i^2 \rho^2 \csch^2 (L_i \rho) -
			2 L_i^3 \rho^3 \csch^2 (L_i \rho) \coth(L_i \rho) -
			L_i \rho \coth(L_i \rho)
		\bigr) \\
		&\quad
		+ \ri \sum_{j=1}^{m} \bigl(
			3 \ell_j^2 \rho^2 \csch^2 (\ell_j \rho) -
			2 \ell_j^3 \rho^3 \csch^2 (\ell_j \rho) \coth(\ell_j \rho) -
			\ell_j \rho \coth(\ell_j \rho)
		\bigr)
		=
		- \ri 6g + \bO(1).
	\end{aligned}
	\]
	Now we set a cut-off $\theta_0$ such that $f'(0) \mkern1mu \theta_0 \to 0$, $f''(0) \mkern1mu \theta_0^2 \to \infty$, $f'''(0) \mkern1mu \theta_0^3 \to 0$ when $g \to \infty$.
	A possible choice is then $\theta_0 = g^{-2/5}$. We thus introduce
	\[
		\mathfrak{F}_{g,\textup{centr}}
		\coloneqq
		\frac{\rho^{-(6g-6+3n)}}{\pi}
		\int_{\mathcal{C}_{\textup{centr}}} \re^{\mkern2mu f(\theta)} \, d\theta,
		\qquad\qquad
		\mathfrak{F}_{g,\textup{tail}}
		\coloneqq
		\frac{\rho^{-(6g-6+3n)}}{\pi}
		\int_{\mathcal{C}_{\textup{tail}}} \re^{\mkern2mu f(\theta)} \, d\theta ,
	\]
	where $\mathcal{C}_{\textup{centr}} \coloneqq \{ \rho \mkern2mu \re^{\ri \theta} \in \CC : |\theta| \leq \theta_0 \}$ and $\mathcal{C}_{\textup{tail}} \coloneqq \{ \rho \mkern2mu \re^{\ri \theta} \in \CC : \theta_0 \le |\theta| \leq \pi/2 \}$. We can now proceed with the three main checks of the saddle-point method: tail pruning, central approximation, and tail completion.

	\textsc{Tail pruning.} Along the tail contour $\mathcal{C}_{\textup{tail}}$ the integrand is negligible as $g \to \infty$. Indeed observe that along the semicircle $\{ \rho \mkern2mu \re^{\ri \theta} \in \CC : |\theta| \le \pi/2 \}$, the integrand is strongly peaked at $\theta = 0$. Thus, along the tail $\mathcal{C}_{\textup{tail}}$, we have $\re^{\mkern2mu f(\theta) - f(0)} = \bO( \re^{|L| \rho \mkern1mu (\cos(g^{-2/5})-1)}) = \bO( \re^{\mkern1mu C g^{-1/5}})$ for some $C > 0$.

	\textsc{Central approximation.} 
	Thanks to the choice $\theta_0 = g^{-2/5}$, we have a quadratic approximation for $f$ along the central contour $\mathcal{C}_{\textup{centr}}$, that is $\re^{\mkern2mu f(\theta)} = \re^{\mkern2mu f(0) - 3\theta^2 g} ( 1 + \bO(g^{-1/5}) )$. Therefore
	\begin{align*}
		\mathfrak{F}_{g,\textup{centr}}
		& =
		\frac{\rho^{-(6g-6+3n)}}{\pi} \, \re^{\mkern2mu f(0)}
		\int_{-\theta_0}^{\theta_0}
		\re^{- 3\theta^2 g} \, d\theta
		\Big( 1 + \bO \big( g^{-1/5} \big) \Big) \\
		& =
		\frac{\rho^{-(6g-6+3n)}}{\pi \sqrt{3g}} \, \re^{\mkern2mu f(0)}
		\int_{-\sqrt{3} g^{1/10}}^{\sqrt{3} g^{1/10}}
		\re^{-\varphi^2} \, d\varphi
		\Big( 1+ \bO \big( g^{-1/5} \big) \Big).
	\end{align*}

	\textsc{Tail completion.} 
	It follows from $\int_a^{+\infty} \re^{-\varphi^2} \, d\varphi = \bO( \re^{-a^2})$ that
	\begin{equation*}
	\begin{split}
		\mathfrak{F}_{g,\textup{centr}}
		& =
		\frac{\rho^{-(6g-6+3n)}}{\pi\sqrt{3g}} \, \re^{\mkern2mu f(0)}
		\int_{-\infty}^{+\infty}
		\re^{-\varphi^2} \, d\varphi
		\Big( 1+ \bO \big( g^{-1/5} \big) \Big) \\
		& =
		\frac{\rho^{-(6g-6+3n)}}{\sqrt{3 \pi g}}
		\Biggl( \prod_{i=1}^{n} \sinh(L_i \rho) \Biggr)
		\Bigg( \prod_{j=1}^{m} \sinh(\ell_j \rho) \Bigg)
		\Big( 1+ \bO \big( g^{-1/5} \big) \Big) \\
		& \sim
		\frac{\rho^{-(6g-6+3n)}}{\sqrt{3 \pi g}}
		\Biggl( \prod_{i=1}^{n} \sinh(L_i \rho) \Biggr)
		\Bigg( \prod_{j=1}^{m} \sinh\left( \frac{\ell_j}{2\mu} \right) \Bigg) .
	\end{split}
	\end{equation*}
	Together with the tail pruning, we have shown that $\mathfrak{F}_{g} \sim \mathfrak{F}_{g,\textup{centr}}$. This completes the proof.
\end{proof}

A direct consequence of the above proposition is an estimate on the symplectic volumes.

\begin{corollary}[{Volume estimates}] \label{cor:VV}
	For any $L = (L_1,\dots,L_n) \in \RR_{>0}^n$ satisfying the boundary scaling assumption \labelcref{eq:bnd:scaling}, we have as $g \to \infty$:
	\begin{equation} \label{eq:Vsim}
		V_{g,n}(L)
		\sim
		\frac{(6g-5+2n)!!}{g! \, 24^g} \, \frac{\rho^{-(6g-6+3n)}}{\sqrt{3\pi g}} \prod_{i=1}^{n} \frac{\sinh(L_i \rho)}{L_i} , 
	\end{equation}
	where $\rho \coloneqq (6g-6+3n)/|L| \sim 1/(2\mu)$. 
	Moreover, for any fixed $\ell = (\ell_1, \dots, \ell_r) \in \RR_{>0}^r$, we have as $g \to \infty$:
	\begin{equation}
		\frac{V_{g-r, n+2r}(L, \ell, \ell)}{V_{g,n}(L)}
		\sim
		2^r \prod_{s=1}^{r} \frac{\cosh(\ell_s/\mu) - 1}{\ell_s^2}.
	\end{equation}
\end{corollary}

\begin{proof}
	The first estimate is a direct combination of \Cref{eq:V} and \Cref{prop:sddl:pnt}. As for the second estimate, the same results imply
	\[
		\frac{V_{g-r, n+2r}(L, \ell, \ell)}{V_{g,n}(L)}
		\sim
		\left( \frac{24g}{6g-5+2n} \right)^r
		\prod_{s=1}^{r} \left( \frac{\sinh(\ell_s/(2\mu))}{\ell_s} \right)^2
		\sim
		2^r \prod_{s=1}^{r} \frac{\cosh(\ell_s/\mu) - 1}{\ell_s^2} .
	\]
	The first estimate follows from $\frac{g!}{(g-r)!} \sim g^r$, while the last estimate follows from $\sinh^2(\frac{x}{2}) = \frac{1}{2}(\cosh(x) - 1)$.
\end{proof}

For a sanity check, consider the one-faced case, i.e.\ when $n = 1$. The $1$-point $\psi$-class intersection numbers admit the closed-form expression $\langle \tau_{3g-2} \rangle_{g,1} = 1/(g! \, 24^g)$, and hence we have
\begin{equation}
	V_{g,1}(L)
	=
	\frac{1}{g! \, 24^g}\frac{L^{6g-4}}{2^{3g-2} (3g-2)!} .
\end{equation}
By Stirling's formula, the above expression is asymptotically equivalent to \labelcref{eq:Vsim} with $n=1$ and $\mu = 1$.

\subsection{Simple, distinct, non-separating closed curves}
We are now ready to prove Claim~\ref{itm:A}. Recall the notation $I = \prod_{i=1}^p [a_i,b_i)^{r_i}$ and $\lambda_{\mu}(\ell) = \frac{\cosh(\ell/\mu) - 1}{\ell}$.

\begin{proposition}\label{prop:EbarNcirc}
	Fix $L = (L_1,\dots,L_n) \in \RR_{>0}^n$ satisfying the boundary scaling assumption. Then
	\begin{equation}
		\lim_{g \to \infty}
		\EE \left[ \bar{N}_{g,L,I}^{\circ} \right]
		=
		\prod_{i=1}^p \left( \int_{a_i}^{b_i} \lambda_{\mu}(\ell) \, d\ell \right)^{r_i} .
	\end{equation}
\end{proposition}

\begin{proof}
	Notice that $\EE[ \bar{N}_{g,L,I}^{\circ}]$ is exactly in the form of the integration formula (\Cref{thm:int:formula}), as all curves under consideration are simple and distinct. Besides, since all curves are non-separating, the associated stable graph $\Gamma$ has $r \coloneqq r_1 + \cdots + r_p$ self-loops attached to a single vertex of genus $g - r$ and $n$ leaves. In this case $|\Aut(\Gamma)| = 2^r$, corresponding to the swapping of half-edges composing the $r$ loops.
	Moreover, in the notation of \Cref{thm:int:formula}, $F = \mathds{1}_I$ is the indicator function of the set $I$. Thus, we find
	\[
		\EE \left[ \bar{N}_{g,L,I}^{\circ} \right]
		=
		\frac{2^{-r}}{V_{g,n}(L)}
		\int_{\RR_{> 0}^r}
			\mathds{1}_I(\ell) \cdot V_{g-r,n+2r}(L, \ell, \ell)
			\prod_{s=1}^r \ell_s \, d\ell_s .
	\]
	The result then follows from the second estimate of \Cref{cor:VV}.
\end{proof}

\subsection{The negligible terms} \label{subsec:negl}
The goal of this section is to prove Claims~\ref{itm:B}--\ref{itm:D}. The main difference compared to the previous section is that the integration formula loses its effectiveness when dealing with non-simple curves.
To overcome this limitation, we consider subsurfaces where the curves are filling.
Similar methods have been applied in e.g.\ \cite{MP19,WX22,LW23,AM}. This strategy is implemented by the introduction of an auxiliary function $M_{A,B;\Gamma}$ depending on parameters $A$ and $B$ and a separating stable graph $\Gamma$. The main reason behind the introduction of such a function is that the functions $N_{g,L,I}^{\asymp}$, $N_{g,L,I}^{\times}$ and $\bar{N}_{g,L,I}^{\circ} - N_{g,L,I}^{\circ}$ can all be bounded by the sum over stable graphs involving the auxiliary $M_{A,B;\Gamma}$ (for some specific choices of $A$ and $B$).

Recall the notation for stable graphs introduced in \Cref{subsec:length}. Denote by $\ms{G}^{\sep}_{g,n}$ the set of stable graphs of type $(g,n)$ with at least two vertices (geometrically, they correspond to separating multicurves). Let $A \in \RR_{\geq 1}$, $B \in \RR_{>0}$ and $\Gamma \in \ms{G}^{\sep}_{g,n}$.
Define $M_{A, B; \Gamma} \colon \cT_{\Sigma}^{\comb}(L) \to \RR$ by setting
\begin{multline} \label{eq:MABGamma}
	M_{A, B; \Gamma}(\GG) \\
	\coloneqq
	\left\vert
	\Set{
		(\gamma_i)_{i=1,\dots,\vert\ms{E}(\Gamma)\vert}
		|
		\substack{
			\displaystyle \text{$\gamma$ multicurve in $\Sigma$ marked by $\Gamma$} \\[.25em]
			\displaystyle \text{with length $\ell_{\GG}(\gamma) \leq B$}
		}
	}
	\right\vert
	\cdot
	2^{|\ms{V}(\Gamma)| + |\ms{E}(\Gamma)|}
	\prod_{v \in \ms{V}_{B}(\Gamma)} A^{6g_v-6+3n_v}.
\end{multline}
Here $\ms{V}_{B}(\Gamma)$ denotes the subset of vertices $v \in \ms{V}(\Gamma)$ such that $v$ is adjacent to no leaf with length larger than $B$. 
For ease of notation, we omit the dependence on the boundary lengths $L = (L_1,\dots,L_n)$. The function $M_{A, B; \Gamma}$ is mapping class group invariant, and it descends to a function on the combinatorial moduli space that will be denoted with the same symbol.

The main technical result, whose proof is postponed to \Cref{sec:proof}, is an estimate for the sum over stable graphs of the expectation value of $M_{A, B; \Gamma}$.

\begin{theorem}[{Main estimate}]\label{thm:sumEM}
	For any $A \geq 1, B > 0$, as $g \to \infty$:
	\begin{equation}
		\sum_{\Gamma \in \ms{G}_{g,n}^{\sep}} \EE\left[ M_{A, B; \Gamma} \right]
		=
		\bO(g^{-1/2}).
	\end{equation}
\end{theorem}

Let us start with Claim~\ref{itm:B}, that is an estimate on the combinatorial length spectrum of simple, distinct, separating cycles.

\begin{proposition}[{Estimate on $N_{g,I,L}^\asymp$}]\label{prop:asympNeqSumM}
	The following bound holds true:
	\begin{equation} \label{eq:asympNeqSumM}
		N_{g,I,L}^\asymp
		\leq
		\sum_{\Gamma \in \ms{G}_{g,n}^{\sep}} M_{1, br; \Gamma}
	\end{equation}
	where $b = \max\Set{b_1,\dots,b_p}$ and $r = r_1+\dots+r_p$. Thus, $\EE[N_{g,I,L}^\asymp] = \bO(g^{-1/2})$.
\end{proposition}

\begin{proof}
	Clearly, $N_{g,I,L}^\asymp \leq \bar{N}_{g,I,L}^\asymp$. On the other hand, any multicurve $\gamma$ counted by $\bar{N}_{g,I,L}^\asymp$ will have an associated stable graph with at least two vertices (since $\gamma$ is separating) and total length bounded by $b_1 r_1 + \cdots + b_p r_p \le br$. Thus,
	\begin{equation*}
		N_{g,I,L}^\asymp
		\leq
		\bar{N}_{g,I,L}^\asymp
		\leq
		\sum_{\Gamma \in \ms{G}_{g,n}^{\sep}} M_{1, br; \Gamma}.
	\end{equation*}
	The value $2^{|\ms{V}(\Gamma)| + |\ms{E}(\Gamma)|}$ appearing in \Cref{eq:MABGamma} is an overestimate in this case. The estimate on the expectation value then follows from \Cref{thm:sumEM}.
\end{proof}

We now proceed with Claim~\ref{itm:C}, that is an estimate on the combinatorial length spectrum of non-simple cycles. Here it will be crucial to consider cycles rather than curves.

\begin{proposition}[{Estimate on $N_{g,I,L}^\times$}] \label{prop:timesNeqSumM}
	The following bound holds true:
	\begin{equation} \label{eq:timesNeqSumM}
		N_{g,L,I}^{\times}
		\leq
		r! \sum_{\Gamma \in \ms{G}_{g,n}^{\sep}} M_{2^{r+1},2br; \Gamma}
	\end{equation}
	where $b = \max\Set{b_1,\dots,b_p}$ and $r = r_1+\dots+r_p$. Thus, $\EE[N_{g,L,I}^{\times}] = \bO(g^{-1/2})$.
\end{proposition}

\begin{proof}
	There are $r!$ ways to arrange $r$ distinct cycles into an ordered $r$-tuple.
	Thus, we find
	\[
		N_{g,L,I}^{\times}
		\leq
		r!
		\left\vert
		\Set{
			\substack{
				\displaystyle \text{unordered collection $\mathtt{C}$ of distinct cycles} \\[.25em]
				\displaystyle \text{with at least one intersection} \\[.25em]
				\displaystyle \text{and length $\ell(\mathtt{C}) \leq br$}
			}
		}
		\right\vert .
	\]
	We claim that the right-hand side is bounded by
	\begin{equation*}
		r! \sum_{\Gamma \in \ms{G}_{g,n}^{\sep}}
			\left\vert
			\Set{
				\substack{
					\displaystyle \text{$\gamma$ unordered multicurve of type $\Gamma$} \\[.25em]
					\displaystyle \text{with length $\ell(\gamma) \leq B$}
				}
			}
			\right\vert
			\cdot
			2^{|\ms{V}(\Gamma)| + |\ms{E}(\Gamma)|}
			\prod_{v \in \ms{V}_{B}(\Gamma)} A^{6g_v-6+3n_v} ,
	\end{equation*}
	where $A = 4^{r}$ and $B = 2br$ are constants. The claim would complete the proof, since the above quantity is bounded by the same expression with ``unordered'' replaced by ``ordered'', which is nothing but \labelcref{eq:MABGamma}.

	The strategy is to assign to any given unordered collection of distinct cycles a separating multicurve.
	We start from an empty multicurve $\gamma$.
	Let $\mathtt{C} = \mathtt{C}_1 + \cdots + \mathtt{C}_r$ be the unordered collection of distinct cycles under consideration, and let $K$ be a connected components of $\mathtt{C}_1 \cup \cdots \cup \mathtt{C}_r$, say $K = \mathtt{C}_{i_1} \cup \cdots \cup \mathtt{C}_{i_k}$. Denote by $\Xi_K$ its tubular neighbourhood. Notice that $\Xi_K$ is a subsurface of $\Sigma$ with boundary, and there are two mutually exclusive situations that can occur.
	\begin{enumerate}
		\item
			$\Xi_K$ is a cylinder.
			In this case, $K = \mathtt{C}_i$ for some $i$ and $\mathtt{C}_i$ is simple.
			We then add $\mathtt{C}_i$ to $\gamma$.

		\item
			$\Xi_K$ has negative Euler characteristics. If a boundary component of $\Xi_K$ is peripheral in $\Sigma$, then we just ignore it. The remaining boundary components form an unordered multicurve, say $\beta = \beta_1 + \cdots + \beta_{m}$.
			Then we add $\beta$ to $\gamma$.
			Note that by construction, the total length of boundary components of $\Xi_K$ is bounded by $2 ( \ell(\mathtt{C}_{i_1}) + \cdots + \ell(\mathtt{C}_{i_k}) ) \leq 2br$.
	\end{enumerate}
	We repeat the same procedure for all connected components of $\mathtt{C}_1 \cup \cdots \cup \mathtt{C}_r$. Note that the same curve may be added several times (at most $3$) during the construction.
	Since we want a primitive multicurve, we will just keep one copy. By construction, all curves added to $\gamma$ are disjoint, the resulting multicurve $\gamma$ is separating, and we have $\ell(\gamma) \leq 2 \ell(\mathtt{C}) \leq 2br = B$. See \Cref{fig:algrthm} for an example.

	\begin{figure}[t]
		\centering
		\begin{subfigure}[t]{.3\textwidth}
			\centering
            \begin{adjustbox}{width=.9\textwidth}
			\begin{tikzpicture}[x=1pt,y=1pt]
				\draw [BrickRed,thick,opacity=.5] (261.892, 720.113) .. controls (263.428, 725.422) and (243.714, 742.711) .. (221.5255, 755.5318) .. controls (199.337, 768.3525) and (174.674, 776.705) .. (170.307, 771.836);
				\draw [BrickRed,thick,opacity=.5] (223.766, 763.719) .. controls (226.1547, 762.3363) and (226.2857, 759.091) .. (224.159, 753.983);
				\draw [BrickRed,thick,opacity=.5] (123.846, 709.586) .. controls (122.694, 717.521) and (152, 752) .. (196.538, 767.681);
				\draw [BrickRed,thick,opacity=.5] (150.184, 649.305) .. controls (143.886, 651.339) and (136, 688) .. (139.782, 732.657);
				\draw [BrickRed,thick,opacity=.5] (219.894, 641.555) .. controls (213.976, 636.433) and (161.602, 643.406) .. (139.77, 686.98);
				
				\draw [BrickRed,thick] (170.307, 771.836) .. controls (165.928, 767.86) and (184.964, 747.93) .. (208.6533, 734.504) .. controls (232.3425, 721.078) and (260.685, 714.156) .. (261.892, 720.113);
				\draw [BrickRed,thick] (201.74, 738.682) .. controls (212.5867, 758.1387) and (219.9287, 766.4843) .. (223.766, 763.719);
				\draw [BrickRed,thick] (190.109, 746.937) .. controls (172, 716) and (123.922, 699.666) .. (123.846, 709.586);
				\draw [BrickRed,thick] (160.445, 718.442) .. controls (176, 680) and (157.797, 645.742) .. (150.184, 649.305);
				\draw [BrickRed,thick] (166.528, 685.906) .. controls (208, 688) and (231.646, 648.377) .. (219.894, 641.555);

				\node [white] at (224.159, 753.983) {$\bullet$};
				\node [white] at (196.538, 767.681) {$\bullet$};
				\node [white] at (139.782, 732.657) {$\bullet$};
				\node [white] at (139.77, 686.98) {$\bullet$};
				\node [BrickRed,opacity=.5] at (224.159, 753.983) {$\bullet$};
				\node [BrickRed,opacity=.5] at (196.538, 767.681) {$\bullet$};
				\node [BrickRed,opacity=.5] at (139.782, 732.657) {$\bullet$};
				\node [BrickRed,opacity=.5] at (139.77, 686.98) {$\bullet$};

				\node [BrickRed] at (201.74, 738.682) {$\bullet$};
				\node [BrickRed] at (190.109, 746.937) {$\bullet$};
				\node [BrickRed] at (160.445, 718.442) {$\bullet$};
				\node [BrickRed] at (166.528, 685.906) {$\bullet$};

				\draw [thick] (267.764, 672.035) .. controls (257.2547, 693.345) and (257.2623, 714.6693) .. (267.787, 736.008);
				\draw [thick] (116.024, 671.599) .. controls (126.6747, 693.1997) and (126.7577, 714.632) .. (116.273, 735.896);
				\draw [thick] (202.638, 621.42) .. controls (214.2127, 641.8067) and (232.0073, 652.572) .. (256.022, 653.716);
				\draw [thick] (181.128, 621.781) .. controls (169.7093, 641.927) and (152.0123, 652.572) .. (128.037, 653.716);
				\draw [thick] (181.327, 786.531) .. controls (169.7757, 766.177) and (152.0157, 755.427) .. (128.047, 754.281);
				\draw [thick] (202.839, 786.274) .. controls (214.2797, 766.0913) and (231.9963, 755.428) .. (255.989, 754.284);
				\draw [thick, line join=round, fill,fill opacity=.05] (202, 780.6667) .. controls (204, 782.6667) and (204, 785.3333) .. (202, 787.3333) .. controls (200, 789.3333) and (196, 790.6667) .. (192, 790.6667) .. controls (188, 790.6667) and (184, 789.3333) .. (182, 787.3333) .. controls (180, 785.3333) and (180, 782.6667) .. (182, 780.6667) .. controls (184, 778.6667) and (188, 777.3333) .. (192, 777.3333) .. controls (196, 777.3333) and (200, 778.6667) .. cycle;
				\draw [thick, line join=round, fill,fill opacity=.05] (202, 620.6667) .. controls (204, 622.6667) and (204, 625.3333) .. (202, 627.3333) .. controls (200, 629.3333) and (196, 630.6667) .. (192, 630.6667) .. controls (188, 630.6667) and (184, 629.3333) .. (182, 627.3333) .. controls (180, 625.3333) and (180, 622.6667) .. (182, 620.6667) .. controls (184, 618.6667) and (188, 617.3333) .. (192, 617.3333) .. controls (196, 617.3333) and (200, 618.6667) .. cycle;
				\draw[thick,shift={(251.679, 657.527)}, rotate=-123.6901, line join=round, fill,fill opacity=.05] (0, 0) .. controls (2, 2) and (2, 4.6667) .. (0, 6.6667) .. controls (-2, 8.6667) and (-6, 10) .. (-10, 10) .. controls (-14, 10) and (-18, 8.6667) .. (-20, 6.6667) .. controls (-22, 4.6667) and (-22, 2) .. (-20, 0) .. controls (-18, -2) and (-14, -3.3333) .. (-10, -3.3333) .. controls (-6, -3.3333) and (-2, -2) .. cycle;
				\draw[thick,shift={(262.773, 733.83)}, rotate=-56.3099, line join=round, fill,fill opacity=.05] (0, 0) .. controls (2, 2) and (2, 4.6667) .. (0, 6.6667) .. controls (-2, 8.6667) and (-6, 10) .. (-10, 10) .. controls (-14, 10) and (-18, 8.6667) .. (-20, 6.6667) .. controls (-22, 4.6667) and (-22, 2) .. (-20, 0) .. controls (-18, -2) and (-14, -3.3333) .. (-10, -3.3333) .. controls (-6, -3.3333) and (-2, -2) .. cycle;
				\draw[thick,shift={(126.773, 653.83)}, rotate=-56.3099, line join=round, fill,fill opacity=.05] (0, 0) .. controls (2, 2) and (2, 4.6667) .. (0, 6.6667) .. controls (-2, 8.6667) and (-6, 10) .. (-10, 10) .. controls (-14, 10) and (-18, 8.6667) .. (-20, 6.6667) .. controls (-22, 4.6667) and (-22, 2) .. (-20, 0) .. controls (-18, -2) and (-14, -3.3333) .. (-10, -3.3333) .. controls (-6, -3.3333) and (-2, -2) .. cycle;
				\draw[thick,shift={(115.679, 737.527)}, rotate=-123.6901, line join=round, fill,fill opacity=.05] (0, 0) .. controls (2, 2) and (2, 4.6667) .. (0, 6.6667) .. controls (-2, 8.6667) and (-6, 10) .. (-10, 10) .. controls (-14, 10) and (-18, 8.6667) .. (-20, 6.6667) .. controls (-22, 4.6667) and (-22, 2) .. (-20, 0) .. controls (-18, -2) and (-14, -3.3333) .. (-10, -3.3333) .. controls (-6, -3.3333) and (-2, -2) .. cycle;
			\end{tikzpicture}
            \end{adjustbox}
			\caption{An embedded metric ribbon graph $\GG$ on a sphere with $6$ holes.}
			\label{fig:algrthm:A}
		\end{subfigure}
		\hspace*{.4cm}%
		\begin{subfigure}[t]{.3\textwidth}
			\centering
            \begin{adjustbox}{width=.9\textwidth}
			\begin{tikzpicture}[x=1pt,y=1pt]
				\draw [ForestGreen, thick, opacity=.5] (122.113, 720.088) .. controls (119.714, 727.9) and (141.857, 745.95) .. (165.9332, 756.3698) .. controls (190.0095, 766.7895) and (216.019, 769.579) .. (223.177, 764.094);
				\draw [ForestGreen, thick, opacity=.5] (260.388, 712.045) .. controls (261.854, 719.958) and (240.927, 737.979) .. (217.4993, 750.415) .. controls (194.0715, 762.851) and (168.143, 769.702) .. (160.04, 763.632);
				\draw [ForestGreen, thick, opacity=.5] (220.099, 641.709) .. controls (213.776, 636.234) and (188.888, 646.117) .. (165.9723, 659.0833) .. controls (143.0565, 672.0495) and (122.113, 688.099) .. (123.576, 696.055);

				\fill [ForestGreen, opacity=.15] (224.295, 644.5582) .. controls (217.4307, 639.5907) and (191.7362, 649.3871) .. (168.1117, 662.4659) .. controls (144.4873, 675.5447) and (122.9328, 691.906) .. (123.9065, 699.4397)-- (122.9328, 691.906) .. controls (121.0625, 684.2603) and (141.7544, 668.5139) .. (164.1834, 655.7037) .. controls (186.6123, 642.8936) and (210.7782, 633.0198) .. (217.4307, 639.5907) -- cycle;
				\fill [ForestGreen, opacity=.3]	(123.9065, 699.4397) .. controls (124.0735, 703.7259) and (159.0576, 701.4547) .. (186.0269, 687.6262) .. controls (212.9961, 673.7976) and (231.9505, 648.4118) .. (224.295, 644.5582) .. controls (224.295, 644.5582) and (222.0069, 642.9024) .. (217.4307, 639.5907) .. controls (224.295, 644.5582) and (207.1841, 668.3016) .. (182.0044, 681.1758) .. controls (156.8246, 694.0499) and (123.576, 696.055) .. (122.9328, 691.906) -- cycle;

				\fill [ForestGreen, opacity=.15] (218.9432, 767.1176) .. controls (213.8747, 771.886) and (200.7903, 771.8541) .. (185.6233, 768.0589) .. controls (175.928, 770.1884) and (168.2714, 770.1637) .. (164.7845, 766.9472) -- (157.1724, 761.9536)
				.. controls (158.8881, 762.9529) and (162.7289, 763.0379) .. (167.9326, 762.3188) .. controls (166.4777, 761.641) and (164.8111, 760.9269) .. (163.1528, 760.1769)
				.. controls (139.8138, 749.6212) and (118.1022, 731.9368) .. (120.9824, 724.1567) -- (123.0086, 715.8996) .. controls (120.9824, 724.1567) and (143.2642, 742.4961) .. (168.246, 752.8668) .. controls (173.6391, 755.1056) and (179.158, 756.9731) .. (184.5608, 758.4855) .. controls (194.1621, 755.7544) and (205.4354, 751.4587) ..
				(215.9477, 746.0641)
				.. controls (240.4474, 733.4916) and (260.8141, 714.9503) ..
				(260.043, 708.4747) -- (261.1058, 716.559) .. controls (263.299, 725.0272) and (242.6894, 742.3148) .. (219.8635, 754.4577) .. controls (214.2562, 757.4407) and (208.515, 760.1133) .. (202.9373, 762.3492) .. controls (214.6952, 764.0893) and (223.9158, 763.7702) .. (227.2434, 761.7008) -- cycle;
				\fill [ForestGreen, opacity=.3] (157.1724, 761.9536) .. controls (151.8821, 759.5353) and (160.3695, 740.4445) .. (175.5, 724.3357) .. controls (146.3253, 712.338) and (122.4642, 718.5918) .. (120.9824, 724.1567) -- (123.0086, 715.8996) .. controls (123.6408, 711.641) and (150.8349, 705.7487) .. (180.9342, 718.9617)  .. controls (185.9697, 714.3471) and (191.5445, 710.2161) .. (197.4578, 707.1218) .. controls (225.032, 692.6929) and (259.9657, 700.8075) .. (260.043, 708.4747) -- (260.8141, 714.9503) .. controls (259.9787, 707.4232) and (227.1257, 699.3226) .. (201.0624, 712.7642)
				.. controls (196.2188, 715.2622) and (191.6097, 718.5041)..	(187.3673, 722.1596) .. controls (205.9578, 731.5016) and (234.2727, 758.5318) .. (227.2434, 761.7008) -- (220.9099, 765.6395) .. controls (227.8482, 761.2808) and (202.1142, 736.6719) .. (182.1589, 727.0274) .. controls (166.6922, 742.6782) and (157.8805, 762.5278) .. (163.8386, 766.2274) -- cycle;
			 
				\draw [ForestGreen, thick] (160.04, 763.632) .. controls (151.96, 759.437) and (171.98, 723.7185) .. (198.962, 709.887) .. controls (225.944, 696.0555) and (259.888, 704.111) .. (260.388, 712.045);
				\draw [ForestGreen, thick] (223.177, 764.094) .. controls (231.842, 759.503) and (203.921, 731.7515) .. (176.7203, 720.8083) .. controls (149.5195, 709.865) and (123.039, 715.73) .. (122.113, 720.088);
				\draw [ForestGreen, thick] (123.576, 696.055) .. controls (123.949, 700.071) and (157.9745, 698.0355) .. (183.9893, 684.6815) .. controls (210.004, 671.3275) and (228.008, 646.655) .. (220.099, 641.709);

				\draw [thick] (267.764, 672.035) .. controls (257.2547, 693.345) and (257.2623, 714.6693) .. (267.787, 736.008);
				\draw [thick] (116.024, 671.599) .. controls (126.6747, 693.1997) and (126.7577, 714.632) .. (116.273, 735.896);
				\draw [thick] (202.638, 621.42) .. controls (214.2127, 641.8067) and (232.0073, 652.572) .. (256.022, 653.716);
				\draw [thick] (181.128, 621.781) .. controls (169.7093, 641.927) and (152.0123, 652.572) .. (128.037, 653.716);
				\draw [thick] (181.327, 786.531) .. controls (169.7757, 766.177) and (152.0157, 755.427) .. (128.047, 754.281);
				\draw [thick] (202.839, 786.274) .. controls (214.2797, 766.0913) and (231.9963, 755.428) .. (255.989, 754.284);

				\draw [thick, line join=round, fill,fill opacity=.05] (202, 780.6667) .. controls (204, 782.6667) and (204, 785.3333) .. (202, 787.3333) .. controls (200, 789.3333) and (196, 790.6667) .. (192, 790.6667) .. controls (188, 790.6667) and (184, 789.3333) .. (182, 787.3333) .. controls (180, 785.3333) and (180, 782.6667) .. (182, 780.6667) .. controls (184, 778.6667) and (188, 777.3333) .. (192, 777.3333) .. controls (196, 777.3333) and (200, 778.6667) .. cycle;
				\draw [thick, line join=round, fill,fill opacity=.05] (202, 620.6667) .. controls (204, 622.6667) and (204, 625.3333) .. (202, 627.3333) .. controls (200, 629.3333) and (196, 630.6667) .. (192, 630.6667) .. controls (188, 630.6667) and (184, 629.3333) .. (182, 627.3333) .. controls (180, 625.3333) and (180, 622.6667) .. (182, 620.6667) .. controls (184, 618.6667) and (188, 617.3333) .. (192, 617.3333) .. controls (196, 617.3333) and (200, 618.6667) .. cycle;
				\draw [thick, shift={(251.679, 657.527)}, rotate=-123.6901, line join=round, fill,fill opacity=.05] (0, 0) .. controls (2, 2) and (2, 4.6667) .. (0, 6.6667) .. controls (-2, 8.6667) and (-6, 10) .. (-10, 10) .. controls (-14, 10) and (-18, 8.6667) .. (-20, 6.6667) .. controls (-22, 4.6667) and (-22, 2) .. (-20, 0) .. controls (-18, -2) and (-14, -3.3333) .. (-10, -3.3333) .. controls (-6, -3.3333) and (-2, -2) .. cycle;
				\draw [thick, shift={(262.773, 733.83)}, rotate=-56.3099, line join=round, fill,fill opacity=.05] (0, 0) .. controls (2, 2) and (2, 4.6667) .. (0, 6.6667) .. controls (-2, 8.6667) and (-6, 10) .. (-10, 10) .. controls (-14, 10) and (-18, 8.6667) .. (-20, 6.6667) .. controls (-22, 4.6667) and (-22, 2) .. (-20, 0) .. controls (-18, -2) and (-14, -3.3333) .. (-10, -3.3333) .. controls (-6, -3.3333) and (-2, -2) .. cycle;
				\draw [thick, shift={(126.773, 653.83)}, rotate=-56.3099, line join=round, fill,fill opacity=.05] (0, 0) .. controls (2, 2) and (2, 4.6667) .. (0, 6.6667) .. controls (-2, 8.6667) and (-6, 10) .. (-10, 10) .. controls (-14, 10) and (-18, 8.6667) .. (-20, 6.6667) .. controls (-22, 4.6667) and (-22, 2) .. (-20, 0) .. controls (-18, -2) and (-14, -3.3333) .. (-10, -3.3333) .. controls (-6, -3.3333) and (-2, -2) .. cycle;
				\draw [thick, shift={(115.679, 737.527)}, rotate=-123.6901, line join=round, fill,fill opacity=.05] (0, 0) .. controls (2, 2) and (2, 4.6667) .. (0, 6.6667) .. controls (-2, 8.6667) and (-6, 10) .. (-10, 10) .. controls (-14, 10) and (-18, 8.6667) .. (-20, 6.6667) .. controls (-22, 4.6667) and (-22, 2) .. (-20, 0) .. controls (-18, -2) and (-14, -3.3333) .. (-10, -3.3333) .. controls (-6, -3.3333) and (-2, -2) .. cycle;
			\end{tikzpicture}
            \end{adjustbox}
			\caption{A collection of intersecting cycles $\texttt{C}$ on $\GG$ and the associated tubular neighbourhood.}
			\label{fig:algrthm:B}
		\end{subfigure}
		\hspace*{.4cm}%
		\begin{subfigure}[t]{.3\textwidth}
			\centering
            \begin{adjustbox}{width=.9\textwidth}
			\begin{tikzpicture}[x=1pt,y=1pt]
				\draw [ForestGreen, thick, opacity=.5] (220.099, 641.709) .. controls (213.776, 636.234) and (188.888, 646.117) .. (165.9723, 659.0833) .. controls (143.0565, 672.0495) and (122.113, 688.099) .. (123.576, 696.055);
				\draw [ForestGreen, thick] (123.576, 696.055) .. controls (123.949, 700.071) and (157.9745, 698.0355) .. (183.9893, 684.6815) .. controls (210.004, 671.3275) and (228.008, 646.655) .. (220.099, 641.709);
				\draw [ForestGreen, thick] (121.9887, 720.5852) .. controls (123.7244, 710.8791) and (149.6143, 703.9385) .. (183.6835, 703.1742) .. controls (217.7526, 702.41) and (260.001, 707.822) .. (261.842, 719.911);
				\draw [ForestGreen, thick, opacity=.5] (261.842, 719.911) .. controls (263.924, 726.875) and (227.962, 739.4375) .. (191.8056, 739.9701) .. controls (155.6492, 740.5027) and (119.2984, 729.0055) .. (121.9887, 720.5852);

				\draw [thick] (267.764, 672.035) .. controls (257.2547, 693.345) and (257.2623, 714.6693) .. (267.787, 736.008);
				\draw [thick] (116.024, 671.599) .. controls (126.6747, 693.1997) and (126.7577, 714.632) .. (116.273, 735.896);
				\draw [thick] (202.638, 621.42) .. controls (214.2127, 641.8067) and (232.0073, 652.572) .. (256.022, 653.716);
				\draw [thick] (181.128, 621.781) .. controls (169.7093, 641.927) and (152.0123, 652.572) .. (128.037, 653.716);
				\draw [thick] (181.327, 786.531) .. controls (169.7757, 766.177) and (152.0157, 755.427) .. (128.047, 754.281);
				\draw [thick] (202.839, 786.274) .. controls (214.2797, 766.0913) and (231.9963, 755.428) .. (255.989, 754.284);
				\draw [thick, line join=round, fill,fill opacity=.05] (202, 780.6667) .. controls (204, 782.6667) and (204, 785.3333) .. (202, 787.3333) .. controls (200, 789.3333) and (196, 790.6667) .. (192, 790.6667) .. controls (188, 790.6667) and (184, 789.3333) .. (182, 787.3333) .. controls (180, 785.3333) and (180, 782.6667) .. (182, 780.6667) .. controls (184, 778.6667) and (188, 777.3333) .. (192, 777.3333) .. controls (196, 777.3333) and (200, 778.6667) .. cycle;
				\draw [thick, line join=round, fill,fill opacity=.05] (202, 620.6667) .. controls (204, 622.6667) and (204, 625.3333) .. (202, 627.3333) .. controls (200, 629.3333) and (196, 630.6667) .. (192, 630.6667) .. controls (188, 630.6667) and (184, 629.3333) .. (182, 627.3333) .. controls (180, 625.3333) and (180, 622.6667) .. (182, 620.6667) .. controls (184, 618.6667) and (188, 617.3333) .. (192, 617.3333) .. controls (196, 617.3333) and (200, 618.6667) .. cycle;
				\draw [thick, shift={(251.679, 657.527)}, rotate=-123.6901, line join=round, fill,fill opacity=.05] (0, 0) .. controls (2, 2) and (2, 4.6667) .. (0, 6.6667) .. controls (-2, 8.6667) and (-6, 10) .. (-10, 10) .. controls (-14, 10) and (-18, 8.6667) .. (-20, 6.6667) .. controls (-22, 4.6667) and (-22, 2) .. (-20, 0) .. controls (-18, -2) and (-14, -3.3333) .. (-10, -3.3333) .. controls (-6, -3.3333) and (-2, -2) .. cycle;
				\draw [thick, shift={(262.773, 733.83)}, rotate=-56.3099, line join=round, fill,fill opacity=.05] (0, 0) .. controls (2, 2) and (2, 4.6667) .. (0, 6.6667) .. controls (-2, 8.6667) and (-6, 10) .. (-10, 10) .. controls (-14, 10) and (-18, 8.6667) .. (-20, 6.6667) .. controls (-22, 4.6667) and (-22, 2) .. (-20, 0) .. controls (-18, -2) and (-14, -3.3333) .. (-10, -3.3333) .. controls (-6, -3.3333) and (-2, -2) .. cycle;
				\draw [thick, shift={(126.773, 653.83)}, rotate=-56.3099, line join=round, fill,fill opacity=.05] (0, 0) .. controls (2, 2) and (2, 4.6667) .. (0, 6.6667) .. controls (-2, 8.6667) and (-6, 10) .. (-10, 10) .. controls (-14, 10) and (-18, 8.6667) .. (-20, 6.6667) .. controls (-22, 4.6667) and (-22, 2) .. (-20, 0) .. controls (-18, -2) and (-14, -3.3333) .. (-10, -3.3333) .. controls (-6, -3.3333) and (-2, -2) .. cycle;
				\draw [thick, shift={(115.679, 737.527)}, rotate=-123.6901, line join=round, fill,fill opacity=.05] (0, 0) .. controls (2, 2) and (2, 4.6667) .. (0, 6.6667) .. controls (-2, 8.6667) and (-6, 10) .. (-10, 10) .. controls (-14, 10) and (-18, 8.6667) .. (-20, 6.6667) .. controls (-22, 4.6667) and (-22, 2) .. (-20, 0) .. controls (-18, -2) and (-14, -3.3333) .. (-10, -3.3333) .. controls (-6, -3.3333) and (-2, -2) .. cycle;
			\end{tikzpicture}
            \end{adjustbox}
			\caption{The resulting separating multicurve $\gamma$. One can check that $\ell_{\GG}(\gamma) \le 2\ell_{\GG}(\texttt{C})$.}
			\label{fig:algrthm:C}
		\end{subfigure}
		\caption{From a collection of intersecting cycles to a separating multicurve.}
		\label{fig:algrthm}
	\end{figure}
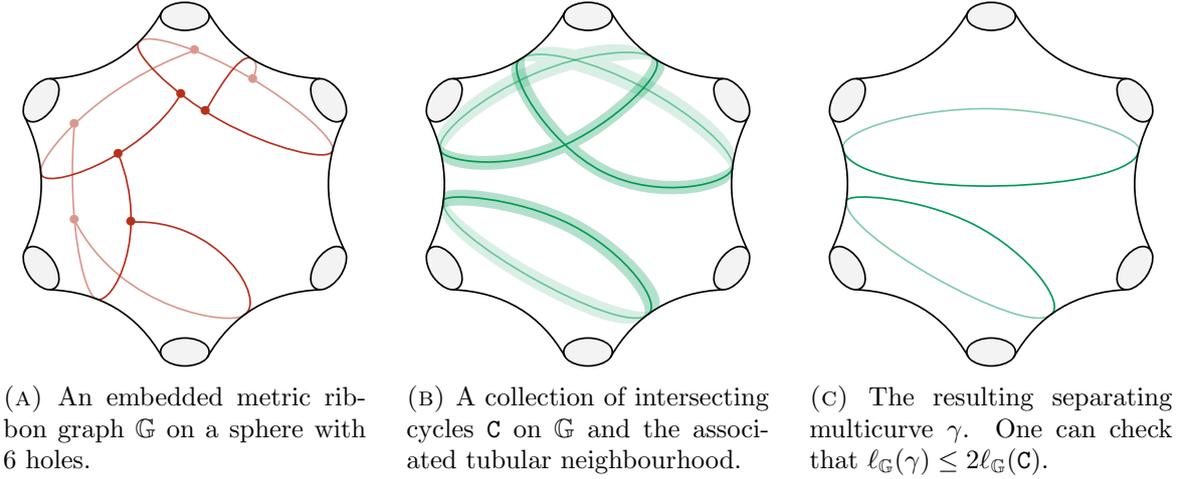

	The drawback of this procedure is that different collections of cycles may yield the same multicurve. Given a (resulting) separating multicurve $\gamma$ of type $\Gamma$, we can estimate the number of different collections of cycles that yield the same $\gamma$ as follows.

	First we bound the number of possible tubular neighbourhoods that can give rise to $\gamma$. Note that such a tubular neighbourhood is the union of cylinders and stable subsurfaces.
	By construction, cylinders correspond to a subset of $\ms{E}(\Gamma)$, while stable subsurfaces correspond to a subset of $\ms{V}_B(\Gamma)$. Hence, the number of tubular neighbourhoods that give rise to the same $\Gamma$ is bounded by $2^{|\ms{E}(\Gamma)| + |\ms{V}_B(\Gamma)|} \le 2^{|\ms{E}(\Gamma)| + |\ms{V}(\Gamma)|}$.

	Next, we bound the number of collections of cycles $\mathtt{C}$ that can give rise to the same tubular neighbourhood, say $\Xi$. Note that a cylinder in $\Xi$ always corresponds to a cycle in $\mathtt{C}$, but a stable subsurface $v$ in $\Xi$ can be induced by different subsets of $\mathtt{C}$. We claim that the number of such subsets is bounded by $A^{6g_v-6+3n_v}$ with $A = 2^{r+1}$.
	Indeed, the number of cycles in the subsurface $v$ is bounded by $2^{6g_v-6+3n_v}$, as there are at most $6g_v-6+3n_v$ edges in $v$ and each edge can either be part of the cycle or not. Moreover, the size of the subset under consideration can vary from $2$ to $r$. Hence, the number of subsets of $\mathtt{C}$ giving rise to $v$ is bounded by $(2^{6g_v-6+3n_v})^2 + \cdots + (2^{6g_v-6+3n_v})^r \leq 2^{(r+1)(6g_v-6+3n_v)}$.
	This gives the desired bound.
\end{proof}

\begin{remark}
	The above proof, which is the only one where considering cycles rather than curves actually matters, generalises to the counting of \emph{$D$-cycles}. A $D$-cycle is an edge-path that visits edges at most $D$ times. In particular, cycles are the same as $1$-cycles. In this case, the above algorithm still holds, but we have to modify the argument for the estimate on the number of different collections of cycles that yield the same tubular neighbourhood. In this case, if a stable subsurface $\Xi$ corresponds to a vertex $v \in \ms{V}_{B}(\Gamma)$, then the number of collections of $D$-cycles in $\Xi$ of size $r$ is bounded by
	\begin{equation}
		\sum_{s=2}^r
			\Bigl(
				(0! + 1! + \cdots + D!)^{6g_v - 6 + 3n_v}
			\Bigr)^s
		\leq
		((D+1)!)^{(r+1)(6g_v - 6 + 3n_v)} .
	\end{equation}
	Thus, \Cref{eq:timesNeqSumM} still holds with with $A = ((D+1)!)^{r+1}$ and $B = 2br$.
\end{remark}

We conclude with a proof of Claim~\ref{itm:D}, that is an estimate on the difference between the combinatorial length spectrum of simple curves and simple cycles.

\begin{proposition}[{Estimate on $\bar{N}_{g,I,L}^{\circ} - N_{g,I,L}^{\circ}$}] \label{lem:simpleNeqSumM}
	The following bound holds true:
	\begin{equation}
		\bar{N}_{g,L,I}^{\circ} - N_{g,L,I}^{\circ}
		\leq
		\sum_{\Gamma \in \ms{G}_{g,n}^{\sep}} M_{1, 3br; \Gamma}
	\end{equation}
	where $b = \max\Set{b_1,\dots,b_p}$ and $r = r_1+\dots+r_p$. Thus, $\EE[\bar{N}_{g,L,I}^{\circ}] - \EE[N_{g,L,I}^{\circ}] = \bO(g^{-1/2})$.
\end{proposition}

\begin{proof}
	Let $\gamma = (\gamma_1, \dots, \gamma_r)$ be an $r$-tuple of simple closed curves which is counted in $\bar{N}^{\circ}_{g,L,I}$ but not in $N^{\circ}_{g,L,I}$. The embedded ribbon graph gives rise to a decomposition of the surface into ribbons (assigned to edges) and discs (assigned to vertices), see \Cref{fig:ribbons:discs}. By taking the geodesic representative of $\gamma$, that is the unique non-backtracking edge-path in the homotopy class, we obtain a collection on non-intersecting segments in each ribbon and a collection of non-intersecting switches in each disc, see \Cref{fig:thick:neigh}. Define the thick neighbourhood of $\gamma$ as the open subset of $\Sigma$ obtained by taking in each ribbon (resp.\ disc) the connected open subset that contains all segments (resp.\ switches), see \Cref{fig:thick:neigh} again. The boundary of the thick neighbourhood of $\gamma$ is a multicurve $\beta$ (possibly with peripheral components that we ignore), that we can make into an ordered multicurve in accordance with an arbitrary order of the (countable) set of closed curves in $\Sigma$. 

	\begin{figure}
		\centering
		\begin{subfigure}[t]{.34\textwidth}
			\centering
            \begin{adjustbox}{width=\textwidth}
			\begin{tikzpicture}

				\draw [Gray,densely dotted,thick] (0,.6) -- (0,-.6);
				\draw [Gray,densely dotted,thick] (3.5,.6) -- (3.5,-.6);

				\draw [BrickRed, thick] (0,0) -- (3.5,0);

				\draw [thick] (0,.6) -- (3.5,.6);
				\draw [thick] (0,-.6) -- (3.5,-.6);

			\begin{scope}[xshift=6cm]

				\draw [Gray,densely dotted,thick] ($(-90:2) + (0:.5)$) -- ($(-90:2) + (180:.5)$);
				\draw [Gray,densely dotted,thick] ($(30:2) + (120:.5)$) -- ($(30:2) + (-60:.5)$);
				\draw [Gray,densely dotted,thick] ($(150:2) + (240:.5)$) -- ($(150:2) + (60:.5)$);

				\draw [BrickRed, thick] (0,0) -- (30:2);
				\draw [BrickRed, thick] (0,0) -- (150:2);
				\draw [BrickRed, thick] (0,0) -- (-90:2);
				\node [BrickRed] at (0,0) {$\bullet$};

				\draw [thick] ($(-90:2) + (0:.5)$) to[out=90,in=210] ($(30:2) + (-60:.5)$);
				\draw [thick] ($(30:2) + (120:.5)$) to[out=210,in=-30] ($(150:2) + (60:.5)$);
				\draw [thick] ($(150:2) + (240:.5)$) to[out=-30,in=90] ($(-90:2) + (180:.5)$);

			\end{scope}
			\end{tikzpicture}
            \end{adjustbox}
			\subcaption{The local decomposition of a surface into ribbons and discs, according to the embedded ribbon graph.}
			\label{fig:ribbons:discs}
		\end{subfigure}
		\hspace{.5cm}%
		\begin{subfigure}[t]{.52\textwidth}
			\centering
            \begin{adjustbox}{width=\textwidth}
			\begin{tikzpicture}

				\draw [Gray,densely dotted,thick] (0,.6) -- (0,-.6);
				\draw [Gray,densely dotted,thick] (3.5,.6) -- (3.5,-.6);

				\fill [ForestGreen,opacity=.3] (0,.4) -- (3.5,.4) -- (3.5,-.4) -- (0,-.4) -- (0,.4);

				\draw [ForestGreen, thick] (0,.2) -- (3.5,.2);
				\draw [ForestGreen, thick] (0,-.2) -- (3.5,-.2);

				\draw [thick] (0,.6) -- (3.5,.6);
				\draw [thick] (0,-.6) -- (3.5,-.6);

			\begin{scope}[xshift=6cm]
				
				\draw [Gray,densely dotted,thick] ($(-90:2) + (0:.5)$) -- ($(-90:2) + (180:.5)$);
				\draw [Gray,densely dotted,thick] ($(30:2) + (120:.5)$) -- ($(30:2) + (-60:.5)$);
				\draw [Gray,densely dotted,thick] ($(150:2) + (240:.5)$) -- ($(150:2) + (60:.5)$);

				\fill [ForestGreen,opacity=.3] ($(-90:2) + (0:.35)$) to[out=90,in=210] ($(30:2) + (-60:.35)$) -- ($(30:2) + (120:.35)$) to[out=210,in=-30] ($(150:2) + (60:.35)$) -- ($(150:2) + (240:.35)$) to[out=-30,in=90] ($(-90:2) + (180:.35)$) -- ($(-90:2) + (0:.35)$);

				\draw [ForestGreen, thick] ($(150:2) + (240:.2)$) to[out=-30,in=90] ($(-90:2) + (180:.15)$);
				\draw [ForestGreen, thick] ($(150:2)$) to[out=-30,in=90] ($(-90:2) + (0:.15)$);
				\draw [ForestGreen, thick] ($(30:2)$) to[out=210,in=-30] ($(150:2) + (60:.2)$);

				\draw [thick] ($(-90:2) + (0:.5)$) to[out=90,in=210] ($(30:2) + (-60:.5)$);
				\draw [thick] ($(30:2) + (120:.5)$) to[out=210,in=-30] ($(150:2) + (60:.5)$);
				\draw [thick] ($(150:2) + (240:.5)$) to[out=-30,in=90] ($(-90:2) + (180:.5)$);

			\end{scope}
			\begin{scope}[xshift=10.5cm]

				\draw [Gray,densely dotted,thick] ($(-90:2) + (0:.5)$) -- ($(-90:2) + (180:.5)$);
				\draw [Gray,densely dotted,thick] ($(30:2) + (120:.5)$) -- ($(30:2) + (-60:.5)$);
				\draw [Gray,densely dotted,thick] ($(150:2) + (240:.5)$) -- ($(150:2) + (60:.5)$);

				\fill [ForestGreen,opacity=.3] ($(30:2) + (-60:.35)$) to[out=210,in=90] ($(-90:2) + (0:.35)$) -- ($(-90:2) + (180:.35)$) to[out=90,in=210] ($(30:2) + (120:.35)$) -- ($(30:2) + (-60:.35)$);

				\draw [ForestGreen, thick] ($(30:2) + (-60:.15)$) to[out=210,in=90] ($(-90:2) + (0:.15)$);
				\draw [ForestGreen, thick] ($(30:2) + (120:.15)$) to[out=210,in=90] ($(-90:2) + (180:.15)$);

				\draw [thick] ($(-90:2) + (0:.5)$) to[out=90,in=210] ($(30:2) + (-60:.5)$);
				\draw [thick] ($(30:2) + (120:.5)$) to[out=210,in=-30] ($(150:2) + (60:.5)$);
				\draw [thick] ($(150:2) + (240:.5)$) to[out=-30,in=90] ($(-90:2) + (180:.5)$);

			\end{scope}
			\end{tikzpicture}
            \end{adjustbox}
			\subcaption{
				The local representation of the thick neighbourhood of a multicurve.
			}
			\label{fig:thick:neigh}
		\end{subfigure}	
		\caption{
			The decomposition of a surface into ribbons and discs, and the thick neighbourhood of a multicurve.
		}
		\label{fig:ribbon:switch:thick:neigh}
	\end{figure}
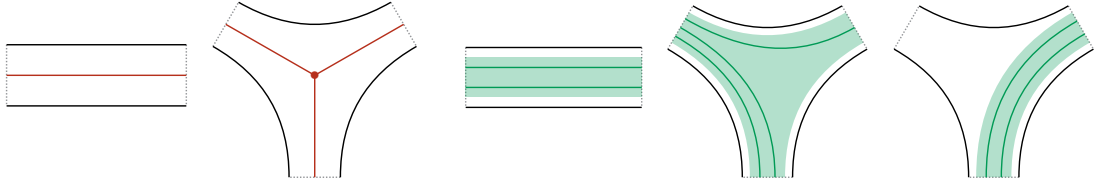

	We claim that $\beta$ is non-trivial and separating. Indeed, if $\beta$ were null-homotopic, then $\gamma$ would be null-homotopic as well (it would be contained within the subsurface bounded by $\beta$). 
	Moreover, if $\beta$ were non-separating, then the thick neighbourhood of $\gamma$ would be a collection of cylinders which contain $\gamma$, in contradiction with the fact that $\gamma$ traverses at least one ribbon more than once.

	Therefore, $\beta$ is a non-trivial separating multicurve. We append to the end of $\gamma$ the curves in $\beta$ that are not already in $\gamma$. Thus, we obtain an ordered separating multicurve, denoted by $\bar{\gamma}$.
	Although we have no control over the number of components of $\bar{\gamma}$, we do know that $\ell(\bar{\gamma}) \leq 3 \ell(\gamma)$. 
	Since the procedure is injective (we simply added curves at the end of the original tuple), we have $\bar{N}_{g,L,I}^{\circ} - N_{g,L,I}^{\circ} \leq \bar{N}_{g,L,\Delta_{\leq 3br}}^{\asymp}$, where $\bar{N}_{g,L,\Delta_{\leq C}}^{\asymp}$ counts separating multicurves whose total length is bounded by $C$.
	Note that $\bar{N}_{g,L,\Delta_{\leq C}}^{\asymp}$ is slightly differ from the counting $\bar{N}_{g,L,I}^{\asymp}$, since in the former case we do not fix the number curves. Nonetheless, the proof-strategy of \Cref{prop:asympNeqSumM} holds with no modifications, hence the claimed bound.
\end{proof}

\section{Proof of \texorpdfstring{\Cref{thm:sumEM}}{the main estimate}}
\label{sec:proof}

The goal of this section is to prove the main estimate on the sum over stable graphs of the expectation value of the auxiliary function $M_{A,B,\Gamma}$. We divide the proof in two parts: an estimate on the single expectation value $\EE[M_{A, B; \Gamma}]$, and an estimate on its sum over all separating stable graphs.

\subsection{Estimates on the auxiliary functions}
To prepare for the estimate on the expectation value of the auxiliary function, we start by giving a basic estimate.

\begin{lemma} \label{lem:KK}
	Let $\Gamma \in \ms{G}_{g,n}^{\sep}$. Then
	\begin{equation}
		\left( \frac{(6g-5+2n)!!}{g! \, 24^g} \right)^{-1}
		\prod_{v \in \ms{V}(\Gamma)}
			\frac{(6g_v+5+2n_v)!!}{g_v! \, 24^{g_v}}
		\leq
		\frac{4^{2- |\ms{V}(\Gamma)| + |\ms{E}(\Gamma)|}}{(2g-2+n)!} \prod_{v \in \ms{V}(\Gamma)} (2g_v - 2 + n_v)!.
	\end{equation}
\end{lemma}

\begin{proof}
	For ease of notation, write $\ms{v} = |\ms{V}(\Gamma)|$ and $\ms{e} = |\ms{E}(\Gamma)|$. The relations
	\[
		\sum_{v \in \ms{V}(\Gamma)} g_v = g-1+\ms{v}-\ms{e} ,
		\qquad
		\sum_{v \in \ms{V}(\Gamma)} n_v = n+2\ms{e} ,
		\quad
		\prod_{i=1}^{m} \binom{ A_i }{ A_{i,1}, \dots, A_{i,r} }
		\leq
		\binom{\sum_{i} A_{i} }{\sum_{i} A_{i,1}, \dots, \sum_{i} A_{i,r}} ,
	\]
	(see \cite[Lemma~2.3]{Agg21} for the last inequality) imply
	\begin{multline*}
		\prod_{v \in \ms{V}(\Gamma)}
			\frac{ (6g_v-5+2n_v)!! }{ g_v! \, 24^{g_v} }
		=
		\prod_{v \in \ms{V}(\Gamma)}
			\binom{ 6g_v-5+2n_v }{ 3g_v-3+n_v, \ 2g_v-2+n_v, \ g_v }
			\frac{ (2g_v-2+n_v)! }{ 2^{3g_v-3+n_v} \, 24^{g_v} } \\
		\leq
		\frac{1}{ 2^{3g-3+n-\ms{e}} \, 24^{g-1+\ms{v}-\ms{e}} }
		\binom{ 6g-6+2n+\ms{v}-2\ms{e} }{ 3g-3+n-\ms{e}, \ 2g-2+n, \ g-1+\ms{v}-\ms{e} }
		\prod_{v \in \ms{V}(\Gamma)} (2g_v-2+n_v)! .
	\end{multline*}
	The connectivity of $\Gamma$ implies $\ms{v} \leq \ms{e} + 1$, hence
	\[
		\frac{(6g-6+2n+\ms{v}-2\ms{e})!}{(6g-5+2n)!}
		\leq
		\frac{1}{(6g-5+2n)^{1- \ms{v} + 2\ms{e}}},
	\]
	\[
		\frac{(3g-3+n)!}{(3g-3+n-\ms{e})!}
		\leq
		(3g-3+n)^{\ms{e}},
		\qquad 
		\frac{g!}{(g-1 + \ms{v} - \ms{e})!}
		\leq
		2(g-1)^{1 - \ms{v} + \ms{e}}.
	\]
	We then conclude that
	\begin{multline*}
		2^{\ms{e}} \, 24^{1- \ms{v} + \ms{e}} \,
		\frac{(6g-6+2n+\ms{v}-2\ms{e})!}{(6g-5+2n)!} \,
		\frac{(3g-3+n)!}{(3g-3+n-\ms{e})!} \,
		\frac{g!}{(g-1+\ms{v}-\ms{e})!} \\
		\leq
		2 \cdot 4^{1-\ms{v}+\ms{e}} \,
		\frac{(6g-6+2n)^{\ms{e}} \, (6g-6)^{1- \ms{v} + \ms{e}}}{(6g-6+2n)^{1- \ms{v} + 2\ms{e}}}
		\leq
		2 \cdot 4^{1 - \ms{v} + \ms{e}}
		\left( \frac{6g-6}{6g-6+2n} \right)^{1 - \ms{v} + \ms{e}}
		\leq
		4^{2 - \ms{v} + \ms{e}} ,
	\end{multline*}
	which proves the lemma.
\end{proof}

\begin{proposition} \label{prop:EMLeq}
	For any constants $A \in \RR_{\ge 1}$ and $B \in \RR_{>0}$, there exists $C = C(A, B, n) > 0$ such that for all $g \gg 0$ and $\Gamma \in \ms{G}_{g,n}^{\sep}$, the following estimate holds:
	\begin{equation}
		\EE\left[ M_{A, B; \Gamma} \right]
		\leq
		\frac{\sqrt{g}}{|\Aut(\Gamma)|} \,
		\frac{C^{|\ms{E}(\Gamma)|}}{(2|\ms{E}(\Gamma)|)!} \,
		\frac{\prod_{v \in \ms{V}(\Gamma)} (2g_v - 2 + n_v)!}{(2g-2+n)!} .
	\end{equation}
\end{proposition}

\begin{proof}
	For ease of notation, write $\ms{V} = \ms{V}(\Gamma)$, $\ms{V}_B = \ms{V}_B(\Gamma)$, $\ms{E} = \ms{E}(\Gamma)$, $\chi = 2g-2+n$, and $\chi_v = 2g_v-2+n_v$ for every $v \in \ms{V}$.
	Applying the integration formula \labelcref{eq:intFor}, we deduce that
	\[
		\EE\left[ M_{A, B; \Gamma} \right]
		=
		\frac{1}{V_{g,n}(L)}
		\frac{2^{|\ms{V}| + |\ms{E}|}}{|\Aut(\Gamma)|}
		\Bigg( \prod_{v \in \ms{V}_B} A^{3\chi_v} \Bigg)
		\int_{\Delta^{\ms{E}}_{\leq B}}
			V_{\Gamma}(L,\ell)
			\prod_{e \in \ms{E}} \ell_e \, d\ell_e 
	\]
	where $\Delta^{\ms{E}}_{\leq B} = \set{ \ell \in \RR_{>0}^{\ms{E}} : \sum_e \ell_e \le B}$. Applying \Cref{thm:uniBd} to bound the intersection numbers in the integrand, we obtain
	\begin{multline*}
		\EE\left[ M_{A, B; \Gamma} \right]
		\leq
		\frac{1}{V_{g,n}(L)}
		\frac{2^{|\ms{V}| + |\ms{E}|}}{|\Aut(\Gamma)|}
		\left( \frac{3}{2} \right)^{n+2|\ms{E}| - |\ms{V}|}
		\left( \prod_{v \in \ms{V}_B} A^{3\chi_v} \right) \\
		\times
		\int_{\Delta^{\ms{E}}_{\leq B}}
			\prod_{v \in \ms{V}} \frac{(6g_v-5+2n_v)!!}{g_v! \, 24^{g_v}}
			[z^{3\chi}]
			\prod_{i=1}^n \frac{\sinh(L_i z)}{L_{\lambda}}
			\prod_{e \in \ms{E}} \left( \frac{\sinh(\ell_e z)}{\ell_e} \right)^2 \ell_e \, d\ell_e .
	\end{multline*}
	We can bound part of the right-hand side with the first estimate from \Cref{cor:VV} and with \Cref{lem:KK}: setting $\rho \coloneqq (6g-6+3n)/|L| \sim 1/(2\mu)$, we have
	\[
		\frac{1}{V_{g,n}(L)}
		\frac{2^{|\ms{V}| + |\ms{E}|}}{|\Aut(\Gamma)|}
		\left( \frac{3}{2} \right)^{n+2|\ms{E}| - |\ms{V}|}
		\prod_{v \in \ms{V}} \frac{(6g_v-5+2n_v)!!}{g_v! \, 24^{g_v}}
		\le
		\frac{\sqrt{g}}{|\Aut(\Gamma)|}
		C_1^{|\ms{E}|}
		\frac{\prod_{v \in \ms{V}} \chi_{v}!}{\chi!}
		\rho^{3\chi}
		\prod_{i=1}^n \frac{L_i}{\sinh(L_i \rho)}
	\]
	for some $C_1 = C_1(n) > 0$ and $g \gg 0$. Here we also used the inequality $|\ms{V}| \le |\ms{E}| + 1$ to write the overall constant as a power of the number of edges. On the other hand, for all $A \in \RR_{\ge 1}$: 
	\begin{multline*}
		[z^{3\chi}]
			\prod_{v \in \ms{V}_B} A^{3\chi_v}
			\prod_{i=1}^n \frac{\sinh(L_i z)}{L_{\lambda}}
			\prod_{e \in \ms{E}} \left( \frac{\sinh(\ell_e z)}{\ell_e} \right)^2 \\
		\leq
		[z^{3\chi}]
			\left( \prod_{v \in \ms{V}_B} \prod_{\lambda \in \ms{\Lambda}(v)}
				\frac{\sinh(L_v A z)}{L_v}
			\right)
			\left( \prod_{v \not\in \ms{V}_B} \prod_{\lambda \in \ms{\Lambda}(v)}
				\frac{\sinh(L_v z)}{L_v}
			\right)
			\prod_{e \in \ms{E}} \left( \frac{\sinh(\ell_e A z)}{\ell_e} \right)^2 .
	\end{multline*}
	Here $\ms{\Lambda}(v)$ denotes the set of leaves attached to a vertex $v$. For any given convergent power series $\sum_{m \geq 0} a_m z^m$ with $a_m \geq 0$, we have $a_k \leq \rho^{-k} \sum_{m \geq 0} a_m \rho^m$ for all choices of $\rho > 0$ (within the disc of convergence). Thus, we find that for all $\ell \in \Delta_{\le B}^{\ms{E}}$:
	\begin{multline*}
		[z^{3\chi}]
			\prod_{v \in \ms{V}_B} A^{3\chi_v}
			\prod_{i=1}^n \frac{\sinh(L_i z)}{L_{\lambda}}
			\prod_{e \in \ms{E}} \left( \frac{\sinh(\ell_e z)}{\ell_e} \right)^2 \\
		\leq
		\rho^{-3\chi}
		\left( \prod_{v \in \ms{V}_B} \prod_{\lambda \in \ms{\Lambda}(v)}
				\frac{\sinh(L_v A \rho)}{L_v}
		\right)
		\left( \prod_{v \not\in \ms{V}_B} \prod_{\lambda \in \ms{\Lambda}(v)}
			\frac{\sinh(L_v \rho)}{L_v}
		\right)
		\prod_{e \in \ms{E}} \left( \frac{\sinh(\ell_e A \rho)}{\ell_e} \right)^2 \\
		\leq
		C_2^{|\ms{E}|}
		\rho^{-3\chi}
		\prod_{i=1}^{n} \frac{\sinh(L_i \rho)}{L_i}
	\end{multline*}
	for some $C_2 = C_2(A,B,n) > 0$ and $g \gg 0$. In the first inequality we chose $\rho = (6g-6+3n)/|L|$ as above, while in the second inequality we used the fact that for all $v \in V_B$ and $\ell \in \Delta^{\ms{E}}_{\leq B}$,
	\[
		\prod_{v \in \ms{V}_B} \prod_{\lambda \in \ms{\Lambda}(v)}
			\frac{\sinh(L_v A \rho)}{L_v}
		\le
		c_1 \prod_{v \in \ms{V}_B} \prod_{\lambda \in \ms{\Lambda}(v)}
			\frac{\sinh(L_v \rho)}{L_v} ,
		\qquad\qquad
		\prod_{e \in \ms{E}} \left( \frac{\sinh(\ell_e A \rho)}{\ell_e} \right)^2
		\le
		c_2^{|\ms{E}|} ,
	\]
	for some $c_i = c_i(A,B,n) > 0$ and $g \gg 0$. To conclude, we apply the identity $\int_{\Delta^{\ms{E}}_{\leq B}} \prod_{e \in \ms{E}} \ell_e \, d\ell_e = \frac{B^{2|\ms{E}|}}{(2|\ms{E}|)!}$.
\end{proof}

\subsection{Summing over all topologies}
We are now ready to estimate the sum of $\EE[M_{A,B;\Gamma}]$ when $\Gamma$ runs over all separating stable graphs. To prepare for it, let us start with a basic estimate on a sum of reciprocals of multinomial coefficients.

\begin{lemma} \label{lem:K}
	Let $2\leq k \leq n$ be integers. The following bound holds true:
	\begin{equation} \label{eq:K}
		\sum_{\substack{(n_1, \dots, n_k) \in \ZZ_{\geq 1}^k \\ n_1 + \cdots + n_k = n}}
			\frac{n_1! \cdots n_k!}{n!}
		\le
		\frac{4}{n} .
	\end{equation}
\end{lemma}

\begin{proof}
	Denote by $F(n,k)$ the left-hand side of \Cref{eq:K}. Let us first prove the claimed bound holds for $k = 2$. We start by rewriting $F(n,2)$ as
	\[
		F(n,2)
		=
		\sum_{\substack{p,q \ge 1 \\ p+q=n}} \frac{p! q!}{n!}
		=
		\frac{2}{n} + \sum_{\substack{p,q \ge 2 \\ p+q=n}} \frac{p! q!}{n!} .
	\]
	In the last sum there are $n-3$ summands, each bounded by $\frac{2}{n(n-1)}$. Thus, the last sum is bounded by $\frac{2(n-3)}{n(n-1)} \le \frac{2}{n}$, proving the claimed bound. Let us now prove that $F(n,k) \le F(n,2)$. We have
	\[
		F(n,k)
		=
		\frac{1}{n!} \sum_{r=2}^{n-(k-2)}
		\sum_{\substack{(n_1, \dots, n_{k-2}) \in \ZZ_{\geq 1}^{k-2} \\ n_1 + \cdots + n_{k-2} = n-r}}
			\frac{n_1! \cdots n_{k-2}!}{n!}
			\sum_{\substack{p,q \ge 1 \\ p+q=r}} p! q! .
	\]
	The first part of the proof implies that the innermost sum is bounded by $4(r-1)!$. Thus, after a relabelling of the index $r$ as $n_{k-1} + 1$, we find
	\[
		F(n,k) \le
		\frac{4}{n!}
		\sum_{\substack{(n_1, \dots, n_{k-1}) \in \ZZ_{\geq 1}^{k-1} \\ n_1 + \cdots + n_{k-1} = n-1}}
			n_1! \cdots n_{k-1}!
		=
		\frac{4}{n} F(n-1,k-1)
		\le F(n-1,k-1) ,
	\]
	where the last inequality holds for $n \ge 4$. By repeatedly applying the above inequality, we find $F(n,k) \le F(n-k+2,2) \le \frac{4}{n}$, hence the thesis. The cases with $n < 4$ can be checked independently.
\end{proof}

We are now ready to prove \Cref{thm:sumEM}. For convenience, the proof makes use of labelled stable graphs, that is a stable graph $\Gamma$ together with bijections $\ms{V}(\Gamma) \to \{ 1, \dots, |\ms{V}(\Gamma)| \}$ and $\ms{H}(\Gamma) \to \{ 1, \dots, |\ms{H}(\Gamma)| \}$.

\begin{proof}[{Proof of \Cref{thm:sumEM}}]
	Let $\Gamma$ be a stable graph of type $(g, n)$ with vertex set $\ms{V}$, half-edges set $\ms{H}$, and edge set $\ms{E}$.
	Denote by $S(\ms{\ms{V}})$ and $S(\ms{H})$ the symmetric group over $\ms{V}$ and $\ms{H}$ respectively.
	Write $S(\Gamma) = S(\ms{V}) \times S(\ms{H})$, and write $\pi$ for the projection that maps a labelled stable graph to its underlying stable graph.
	The group $S(\Gamma \mkern1mu )$ acts on $\pi^{-1}(\Gamma)$ by permuting the labels, and the stabiliser of any element in $\pi^{-1}(\Gamma)$ is isomorphic to $\Aut(\Gamma)$.
	Thus, we have $|\pi^{-1}(\Gamma)| = |S(\Gamma)| / |\Aut(\Gamma)|$, so for any function $f$ defined on the set of labelled stable graphs that is constant along the fibres of $\pi$, we have
	\[
		\sum_{\Gamma \in \ms{G}_{g,n}^{\sep}} \frac{f(\Gamma)}{|\Aut(\Gamma)|}
		=
		\sum_{\Gamma}
			\frac{f(\Gamma)}{|\ms{V}|! \, |\ms{H}|!}.
	\]
	The sum on the right-hand side runs over the set of labelled separating stable graphs of type $(g, n)$.
	In light of \Cref{prop:EMLeq}, we consider the following choice of $f$:
	\[
		f(\Gamma)
		\coloneqq
		\sqrt{g} \,
		\frac{C^{|\ms{E}|}}{(2|\ms{E}|)!}
		\frac{\prod_{v \in \ms{V}} (2g_v-2+n_v)!}{(2g-2+n)!}.
	\]
	Note that $f(\Gamma)$ depends only on $\ms{e} = |\ms{E}|$, the genus decoration $(g_v)_{v \in \ms{V}}$, and the valency decoration $(n_v)_{v \in \ms{V}}$, so we can write $f(\ms{e}, (g_v), (n_v))$ for $f(\Gamma)$.

	Given $m,k \in \ZZ_{\geq 1}$, we write
	\begin{align*}
		\ms{C}_0(m,k)
		& \coloneqq \Set{ (m_1, \dots, m_k) \in \ZZ_{\geq 0}^k | m_1 + \cdots + m_k = m }, \\
		\ms{C}_1(m,k)
		& \coloneqq \Set{ (m_1, \dots, m_k) \in \ZZ_{\geq 1}^k | m_1 + \cdots + m_k = m }.
	\end{align*}
	We first claim that
	\begin{multline*}
		\sum_{\Gamma}
			\frac{f(\Gamma)}{|\sf{V}|! \, |\ms{H}|!}
		\leq
		\sum_{\substack{ \ms{v} \ge 2, \, \ms{e} \ge 1 \\ \ms{v} \leq \ms{e}+1 }} 
			\frac{1}{\ms{v}! \, (2\ms{e})!}
				\sum_{\substack{ s \ge 0, \, t \geq 1 \\ s+t = \ms{e} }}
				\sum_{\substack{
					\ms{g} \in \ms{C}_0(g-1-\ms{e}+\ms{v}, \ms{v}) \\
					\ms{n} \in \ms{C}_0(n, \ms{v}) \\
					\ms{s} \in \ms{C}_0(s, \ms{v}), \ms{t} \in \ms{C}_1(2t, \ms{v})
				}} \\
				\times
					\binom{n}{n_1, \dots, n_{\ms{v}}}
					\binom{2\ms{e}}{2s_1+t_1, \dots, 2s_{\ms{v}}+t_{\ms{v}}}
					\left( \prod_{i=1}^{\ms{v}} \binom{2s_i+t_i}{2s_i} (2s_i-1)!! \right)
					(2t-1)!!
					\,
					f(\ms{e},\ms{g},\ms{n}+2\ms{s}+\ms{t}) .
	\end{multline*}
	Indeed, any labelled stable graph of type $(g,n)$ having $\ms{v} \geq 2$ vertices and $\ms{e} \geq 1$ edges, out of which $s \geq 0$ are self-loops and $t \ge 1$ are not, can be constructed as follows.
	Fix $\ms{v}$ vertices and decorate them with their genus $\textsf{g} \in \ms{C}_0(g-1-\ms{e}+\ms{v}, \ms{v})$. Fix a splitting of the leaves, that is $\ms{n} \in \ms{C}_0(n, \ms{v})$ and distribute the labels. This can be achieved in $\binom{n}{n_1, \dots, n_{\textsf{v}}}$ different ways.
	Second, fix a splitting of the half-edges that are not leaves as self-loops and non-self-loops, that is $\textsf{s} \in \ms{C}_0(s, \ms{v})$ and $\textsf{t} \in \ms{C}_1(2t, \ms{v})$. Attach $2s_i + t_i$ half-edges to the $i$-th vertex; there are $\binom{2\ms{e}}{2s_1 + t_1, \dots, 2s_{\textsf{v}} + t_{\textsf{v}}}$ ways to label them.
	Third, among the $2s_i + t_i$ half-edges attached to the $i$-th vertex, choose $2s_i$ of them and pair them up to form $s_i$ self-loops; there are $\binom{2s_i + t_i}{2s_i} (2s_i-1)!!$ ways to do so.
	Finally, we tie the remaining $2t$ half-edges two-by-two to form $t$ edges connecting distinct vertices;
	there are at most $(2t-1)!!$ ways to do so. This is an overestimate, since we may create self-loops by connecting the last half-edges, or create a disconnected graph, or create a graph where the stability condition does not hold. Nonetheless, this proves the above claim.

	Denote $\chi = 2g-2+n$. With our choice of $f$, we find:
	\begin{align*}
		\sum_{\Gamma}
			\frac{f(\Gamma)}{|\ms{V}|! \, |\ms{H}|!}
		& \leq
		\sqrt{g}
		\sum_{\substack{ \ms{v} \ge 2, \, \ms{e} \ge 1 \\ \ms{v} \leq \ms{e}+1 }} 
			\frac{C^{\ms{e}}}{\ms{v}! \, (2\ms{e})!}
			\sum_{\substack{ s \ge 0, \, t \geq 1 \\ s+t = \ms{e} }}
				\frac{n! \, (2t-1)!!}{2^s \, \chi!}
				\mkern-25mu \sum_{\substack{
					\ms{g} \in \ms{C}_0(g-1-\ms{e}+\ms{v}, \ms{v}) \\
					\ms{n} \in \ms{C}_0(n, \ms{v}) \\
					\ms{s} \in \ms{C}_0(s, \ms{v}), \ \ms{t} \in \ms{C}_1(2t, \ms{v})
				}}
					\prod_{i=1}^{\ms{v}} \frac{(2g_i-2+n_i+2s_i+t_i)!}{n_i! \, s_i! \, t_i!} \\
		& \leq
		\sqrt{g}
		\sum_{\substack{ \ms{v} \ge 2, \, \ms{e} \ge 1 \\ \ms{v} \leq \ms{e}+1 }}  
			\frac{C^{\ms{e}}}{\ms{v}! \, (2\ms{e})!}
			\sum_{\substack{ s \ge 0, \, t \geq 1 \\ s+t = \ms{e} }}
				\frac{n! \, (2t-1)!!}{2^s \, \chi!}
				\frac{\ms{v}^{n+s+t}}{n! \, s! \, t!}
				\sum_{\ms{x} \in \ms{C}_1(\chi, \ms{v})}
					\prod_{i=1}^{\ms{v}} \ms{x}_i \\
		& \leq
		3\sqrt{g}
		\sum_{\substack{ \ms{v} \ge 2, \, \ms{e} \ge 1 \\ \ms{v} \leq \ms{e}+1 }}  
			\frac{\ms{v}^n}{\ms{v}!} \frac{(2\ms{v}C)^{\ms{e}}}{(2\ms{e})!}
				\frac{1}{\chi !}
				\sum_{\ms{x} \in \ms{C}_1(\chi, \ms{v})}
					\prod_{i=1}^{\ms{v}} \ms{x}_i.
	\end{align*}
	The second last inequality follows from the multinomial theorem, while the last inequality follows form $2^{-s} \frac{(2t-1)!!}{t!} \leq  2^{t-s} \leq 2^{\ms{e}}$ and $\sum_{s+t = \ms{e}} \frac{1}{s!} \le \exp(1) \le 3$. By \Cref{lem:K}, the above quantity is bounded by
	\[
		\frac{12}{\chi} \sqrt{g}
		\sum_{\substack{ \ms{v} \ge 2, \, \ms{e} \ge 1 \\ \ms{v} \leq \ms{e}+1 }}  
			\frac{\ms{v}^n}{\ms{v}!} \frac{(2\ms{v}C)^{\ms{e}}}{(2\ms{e})!}
		\leq
		\frac{12}{\chi} \sqrt{g}
		\sum_{\ms{v}=2}^{\infty}
			\frac{\ms{v}^n}{\ms{v}!}
			\exp(2\ms{v}C)
		\leq
		\frac{12 C'}{\chi} \sqrt{g}
		=
		\bO(g^{-1/2})
	\]
	for some $C' = C'(A,B,n) > 0$.
	Now the result follows from \Cref{prop:EMLeq}.
\end{proof}

\section{What is wrong with closed curves} \label{sec:wrong}

Contrary to the hyperbolic case, in the metric ribbon graph setting we consider the length spectrum of closed curves satisfying a specific condition: the cycle condition. The goal of this section is to explain why this is the case. In a nutshell, the reason is the absence of a collar lemma for metric ribbon graphs.

In \cite{Bas93,Bas13}, Basmajian has shown that the length of any closed geodesic with self-intersection number $k$ on any hyperbolic surface is bounded below by some universal constant $M_k$, and $M_k \to \infty$ as $k \to \infty$. In particular, if a closed geodesic intersects itself many times, then it cannot be very long. Basmajian's proof relies on the (generalised) collar lemma, which says roughly that a short closed geodesic on a hyperbolic surface has a large tubular neighbourhood which is a topological cylinder.

In the context of metric ribbon graphs, the collar lemma fails dramatically.
The distinction lies in the fact that metric ribbon graphs, unlike hyperbolic surfaces characterised by a constant non-zero sectional curvature, permit scaling. In particular, short curves of high topological complexity ($\omega(1)$ self-intersections or $\omega(1)$ intersections between two curves) can exist within a ribbon graph of low topological complexity, as we go deep into the thin part of the moduli space. This idea was used to establish the $L^p$-integrability of the combinatorial unit ball of measured foliations in \cite{BCDGW22}, which exhibit different behaviour compared to its hyperbolic analogue \cite{AA20}. In this section, we use the same idea to study the $L^p$-integrability of the functions $\bar{N}_{g, L, [a, b)}^{\circ}$ and $\bar{N}_{g, L, [a, b)}$.

Let us start with an estimate for the case of a one-holed torus.

\begin{lemma} \label{lem:NbarcircGeq}
	Fix $0 \leq a < b$.
	\begin{itemize} \setlength\itemsep{.5em}
		\item
		There exist $\epsilon^{\circ}, \delta^{\circ}, C^{\circ} > 0$ depending only on $a$ and $b$ such that, for any $L \leq \epsilon^{\circ}$ and $G \in \cM_{1, 1}^{\comb}(L)$ trivalent with edge lengths in $[(1/3 - \delta^{\circ}) L, (1/3 + \delta^{\circ}) L]$, we have
		\begin{equation}
			\bar{N}_{1, L, [a, b)}^{\circ}(G)
			\geq
			\frac{C^{\circ}}{L^2} .
		\end{equation}

		\item
		There exist $\epsilon, \delta, C, c > 0$ depending only on $a$ and $b$ such that, for any $L \leq \epsilon$ and $G \in \cM_{1, 1}^{\comb}(L)$ trivalent with edge lengths in $[(1/3 - \delta) L, (1/3 + \delta) L]$, we have
		\begin{equation}
			\bar{N}_{1, L, [a, b)}(G)
			\geq
			C \, L \, \re^{c/L}.
		\end{equation}
	\end{itemize}
\end{lemma}

\begin{proof}
	Let us proceed with the first claim. Let $\epsilon^{\circ} > 0$, $0 < \delta^{\circ} < 1/3$, and $G \in \cM_{1, 1}^{\comb}(\epsilon^{\circ})$ trivalent with edge lengths in $[(1/3 - \delta^{\circ}) L, (1/3 + \delta^{\circ}) L]$.
	The fundamental group of a torus with a single boundary component is a free group of rank $2$.
	We can choose a generating set $\{ \alpha, \beta \}$ such that $\alpha$ and $\beta$ traverse exactly two edges.
	A cyclically reductive word in $\{ \alpha, \beta \}$ of $m$ letters has length (with respect to $G$) between $2(1/3-\delta^{\circ})\epsilon^{\circ} m$ and $2(1/3+\delta^{\circ}) \epsilon^{\circ} m$. If $\epsilon^{\circ}$ and $m$ satisfy
	\[
		\frac{at}{2(1/3 - \delta^{\circ})} \leq \epsilon^{\circ} m < \frac{bt}{2(1/3+\delta^{\circ})}
		\qquad\text{for some }t > 0,
	\]
	which is possible whenever $\delta^{\circ} < (b-a)/(3a+3b)$. Then every cyclically reductive word of $m$ letters has length (with respect to $G$) in $[ta, tb)$, and we find
	\[
		\frac{a}{1/3 - \delta^{\circ}} \frac{t}{\epsilon^{\circ}} \leq m < \frac{b}{1/3 + \delta^{\circ}} \frac{t}{\epsilon^{\circ}}.
	\]
	It follows from \cite[Theorem~2.7]{MR95} that for $\epsilon^{\circ}$ small enough
	\[
		\bar{N}^{\circ}_{1, \epsilon^{\circ}, [ta, tb)}(G)
		\geq
		\frac{1}{4\zeta(2)} \left( \frac{b}{1/3 + \delta^{\circ}} - \frac{a}{1/3 - \delta^{\circ}} \right)^2 \left( \frac{t}{\epsilon^{\circ}} \right)^2.
	\]
	Therefore, for any $L \leq \epsilon^{\circ}$, we have
	\[
		\bar{N}^{\circ}_{1,L,[a, b)}(G)
		=
		\bar{N}^{\circ}_{1,\epsilon^{\circ},[\frac{\epsilon^{\circ}}{L} a, \frac{\epsilon^{\circ}}{L} b)}
		\left( \tfrac{\epsilon^{\circ}}{L} \cdot G \right)
		\geq 
		\frac{1}{4\zeta(2)} \left( \frac{b}{1/3 + \delta^{\circ}} - \frac{a}{1/3 - \delta^{\circ}} \right)^2 \frac{1}{L^2}
	\]
	as claimed.

	As for the second inequality, the argument is similar. It can be shown that the number of conjugacy classes in a torus with a single boundary component representing primitive closed curves of $m$ letters is asymptotically equivalent to $3^m/m$. Thus, for any small enough $\epsilon$, we have
	\[
		\bar{N}_{1, \epsilon, [ta, tb)}(G)
		\geq
		C (\epsilon / t) \exp(c t/\epsilon)
	\]
	for some $C, c > 0$.
\end{proof}

\begin{proof}[Proof of \Cref{thm:turtle:neck}]
	The volume of subset of $\cM_{1,1}^{\comb}(\ell)$ consisting of those metric ribbon graphs where every edge has length in $[(1/3-\delta^{\circ}) \ell, (1/3+\delta^{\circ}) \ell]$ is $(2\delta^{\circ})^3 \, V_{1,1}(\ell)$.
	Thus, it follows from the first claim of \Cref{lem:NbarcircGeq} that
	\[
		\EE \left[ (\bar{N}_{g, L, [a, b)}^{\circ})^{k} \right]
		\geq
		\frac{(C^{\circ})^k (2\delta^{\circ})^3}{V_{g,n}(L)}
		\int_0^{\epsilon^{\circ}}  \ell^{-2k} \cdot V_{1,1}(\ell) \cdot V_{g-1,n+1}(L_1, \dots, L_n, \ell) \, \ell \, d\ell.
	\]
	The right-hand side blows up if $k > 3/2$. 
	Similarly, from the second claim of \Cref{lem:NbarcircGeq}:
	\[
		\EE \left[ \bar{N}_{g, L, [a, b)} \right]
		\geq
		\frac{C (2\delta)^3}{V_{g,n}(L)}
		\int_0^{\epsilon} \ell \, \re^{C/\ell} \cdot V_{1,1}(\ell) \cdot V_{g-1,n+1}(L_1, \dots, L_n, \ell) \, \ell \, d\ell
		=
		\infty .
	\]
	This completes the proof.
\end{proof}

\section{Numerical evidence}\label{sec:numerics}

Let us briefly outline the simulation for the unicellular case, i.e.\ when $n = 1$. We fix the unique boundary length as $L = 12g$. A random ribbon graph is almost surely trivalent, and in the unicellular case, each trivalent ribbon graph has an equal chance of being selected. Hence, $\bm{G}_{g,L}$ can be sampled in two steps: firstly, we sample a trivalent one-faced ribbon graph, and secondly, we endow it with a uniformly random metric. Once a random metric graph is generated, we can count cycles to get the bottom part of its length spectrum.

\subsection*{Generating random unicellular metric maps} 
The first step can be achieved as follows. Based on the work of Chapuy, Féray, and Fusy \cite{CFF13}, the generation of a random combinatorial unicellular trivalent map can be reduced to generating a random plane trivalent tree and a random permutation that satisfies certain conditions.
The generation of random trees can be accomplished recursively using Rémy's algorithm \cite{Rem85}. Starting with the tree having one vertex and two leaves, at each step, an edge is uniformly chosen, a vertex is added at its midpoint, and a leaf is attached to the new vertex. Once the tree accumulates $6g-3$ edges, we randomly merge its leaves three-by-three, resulting in a random trivalent unicellular map of genus $g$. To conclude, each edge of the resulting graph is endowed with a length sampled according to a Dirichlet distribution of parameters $(1^{6g-3})$. See \Cref{fig:graph} for examples of random unicellular maps of genus $64$.

\begin{figure}[h]
	\begin{center}
	\includegraphics[width=.4\textwidth]{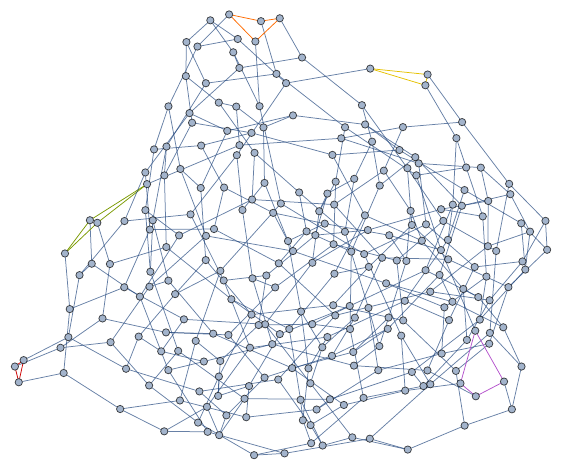}%
	\hspace{5mm}
	\includegraphics[width=.4\textwidth]{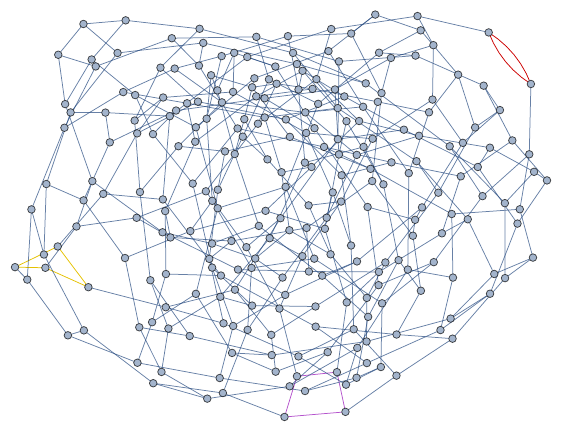}
	\caption{
		The graphs underlying two random unicellular maps of genus $64$. The highlighted cycles include all cycles with at most $4$ edges.
	}
	\label{fig:graph}
	\end{center}
\end{figure}

\subsection*{Counting cycles} 
What remains now is to find all cycles of length falling within a fixed interval $[a,b)$. We performed this search using the \texttt{FindCycle} function in \texttt{Mathematica} and recording the lengths of the resulting cycles.

\subsection*{The simulation} 
For the simulations presented in \Cref{fig:simulation}, we chose $[a,b) = [0,4)$. For efficiency reasons, our search was limited to cycles with at most $12$ edges. Since edge-lengths are, on average, equal to $1$, only few cycles are missed in the search. Nonetheless, this accounts for why the histogram is slightly below the theoretical prediction.

\printbibliography
\end{document}